\documentclass[12pt,twoside,reqno]{amsart}
\usepackage[colorlinks=true,citecolor=blue]{hyperref}
\usepackage{mathptmx, amsmath, amssymb, amsfonts, amsthm, mathptmx, enumerate, color}
\usepackage{graphicx}
\usepackage{epstopdf}
\usepackage{color}

\newtheorem{theorem}{Theorem}[section]

\newtheorem{lemma}{Lemma}[section]

\theoremstyle{definition}
\newtheorem{definition}{Definition}[section]

\newcommand{\tcr}{\textcolor{red}}
\numberwithin{equation}{section}

\begin{document}
\setcounter{page}{1}

\vspace*{1.0cm}
\title[Critical Kirchhoff-Choquard equations  on the Heisenberg group]
{Existence and  multiplicity of solutions for critical
Kirchhoff-Choquard equations involving the fractional $p$-Laplacian
on the Heisenberg group}
\author[S. Bai, Y. Song, D.D. Repov\v{s}]{\bf Shujie Bai$^{1}$, Yueqiang Song$^{1,*}$, Du\v{s}an D. Repov\v{s}$^{2,3,4}$}
\maketitle \vspace*{-0.6cm}

\begin{center}
{\footnotesize {

$^1$College of Mathematics, Changchun Normal
University,   Changchun, 130032,  P.R. China\\
E-mail addresses: annawhitebai@163.com (S. Bai), songyq16@mails.jlu.edu.cn (Y. Song)\\
$^2$ Faculty of Education, University of Ljubljana, Ljubljana, 1000, Slovenia\\
$^3$ Faculty of Mathematics and Physics,  University of Ljubljana,  Ljubljana, 1000, Slovenia\\
$^4$ Institute of Mathematics, Physics and Mechanics, Ljubljana, 1000, Slovenia\\
E-mail address: dusan.repovs@guest.arnes.si 

 }}\end{center}

\vskip 4mm {\small\noindent {\bf Abstract.}
In this paper, we study existence and multiplicity of solutions  for the
following Kirchhoff-Choquard type equation involving the fractional
$p$-Laplacian  on the Heisenberg group:
\begin{equation*}
\begin{array}{lll}
M(\|u\|_{\mu}^{p})(\mu(-\Delta)^{s}_{p}u+V(\xi)|u|^{p-2}u)=
f(\xi,u)+\int_{\mathbb{H}^N}\frac{|u(\eta)|^{Q_{\lambda}^{\ast}}}{|\eta^{-1}\xi|^{\lambda}}d\eta|u|^{Q_{\lambda}^{\ast}-2}u
&\mbox{in}\  \mathbb{H}^N, \\
\end{array}
\end{equation*}
where $(-\Delta)^{s}_{p}$ is the fractional $p$-Laplacian on the
Heisenberg group $\mathbb{H}^N$, $M$ is the Kirchhoff function,
$V(\xi)$ is the potential function, $0<s<1$, $1<p<\frac{N}{s}$,
$\mu>0$, $f(\xi,u)$ is the nonlinear function, $0<\lambda<Q$,
$Q=2N+2$, and $Q_{\lambda}^{\ast}=\frac{2Q-\lambda}{Q-2}$ is the
Sobolev critical exponent. Using the Krasnoselskii genus theorem,
the existence of infinitely many solutions is obtained if $\mu$ is
sufficiently large. In addition, using the fractional version of the
concentrated compactness principle, we prove that problem has
$m$ pairs of solutions if $\mu$ is sufficiently small. As far as
we know, the results of
our study are new even in the Euclidean case.

\noindent {\bf Keywords.}
Kirchhoff-Choquard type equations; Heisenberg
group; Fractional concentration-compactness principle; Krasnoselskii
genus.

\noindent {\bf Math. Subj. Classif.}
35J20; 35R03: 46E35.

\renewcommand{\thefootnote}{}
\footnotetext{$^*$Corresponding author.}

\section{Introduction}

In this paper, we study the existence and multiplicity of solutions
for the following Kirchhoff-Choquard type equation involving the
fractional $p$-Laplacian  on the Heisenberg group of the form:
\begin{equation}\label{1.1}
\begin{array}{lll}
M(\|u\|_{\mu}^{p})(\mu(-\Delta)^{s}_{p}u+V(\xi)|u|^{p-2}u)=
f(\xi,u)+\int_{\mathbb{H}^N}\frac{|u(\eta)|^{Q_{\lambda}^{\ast}}}{|\eta^{-1}\xi|^{\lambda}}d\eta|u|^{Q_{\lambda}^{\ast}-2}u
&\mbox{in}\  \mathbb{H}^N, \\
\end{array}
\end{equation}
where $(-\Delta)^{s}_{p}$ is the fractional $p$-Laplacian on the
Heisenberg group $\mathbb{H}^N$, $M$ is the Kirchhoff function,
$V(\xi)$ is the potential function, $0<s<1$, $1<p<\frac{N}{s}$,
$f(\xi,u)$ is the nonlinear function, $\mu>0$, $0<\lambda<Q$,
$Q=2N+2$, and $Q_{\lambda}^{\ast}=\frac{2Q-\lambda}{Q-2}$ is the
Sobolev critical exponent.

Suppose that the Kirchhoff function $M$ and potential function $V$ satisfy the following assumptions:\\
$(M)$ $M\in C(\mathbb{R},\mathbb{R})$ and there exist $\tau\in
(1,\frac{Q_\lambda^\ast}{p}]$ and $0<m_{0}\leq m_{1}$ satisfying
$$m_{0}t^{\tau-1}\leq M(t)\leq m_{1}t^{\tau-1}\quad \mbox{for every}  \ t\in \mathbb{R}_{0}^{+},$$
that is, $M$ is non-degenerate.
\begin{itemize}
\item[$(V_1)$] $V(\xi)\in C(\mathbb{H}^N,\mathbb{R})$ with $V(\xi)\geq \min V(\xi)=0$;

\item[$(V_2)$] there exists $R>0$ satisfying
$\lim_{|\eta|\rightarrow\infty}\mbox{meas}(\{\xi\in B_{R}(\eta):V(\xi)\leq
c\})=0$ for every $c>0$, where meas$(\cdot)$ denotes the Lebesgue
measure on $\mathbb{H}^N$.
\end{itemize}
The nonlinearity $f(\cdot,\cdot): \mathbb{R}^N\times
\mathbb{R}\rightarrow\mathbb{R}$ is a Carath\'{e}odory function,
which requires different assumptions for critical growth and
subcritical growth, respectively. For the case of critical exponent
$\tau=\frac{Q_\lambda^\ast}{p}$, $f$ satisfies the following
assumptions:
\begin{itemize}
\item[ $(f_{1})$] there exists $q\in(p,Q_\lambda^\ast)$
such that for every $\varepsilon>0$ there exists $C_{\varepsilon}>0$
satisfying
$$|f(\xi,t)|\leq p\varepsilon|t|^{p-1}+qC_{\varepsilon}|t|^{q-1}\quad \mbox{a.e.}  \ \xi\in \mathbb{H}^{N}\quad\mbox{and}\quad \mbox{for every}\ t\in\mathbb{R};$$

\item[$(f_2)$] there exist $a_{1}>0$, $q_{1}\in(p,Q_\lambda^\ast)$ satisfying
$$F(\xi,t)=\int_{0}^{t}f(\xi,s)ds\geq a_{1}|t|^{q_{1}}\quad \mbox{a.e.}  \ \xi\in \mathbb{H}^{N}\quad\mbox{and}\quad \mbox{for every}\ t\in\mathbb{R}.$$
\end{itemize}

For the case of subcritical exponent
$\tau\in(1,\frac{Q_\lambda^\ast}{p})$, the following conditions
should be satisfied for $f$:
\begin{itemize}
\item[ $(f_{1})^{'}$] there exists
$q\in(\tau p,Q_\lambda^\ast)$ such that for every $\varepsilon>0$
there exists $C_{\varepsilon}>0$ satisfying
$$|f(\xi,t)|\leq \tau p\varepsilon|t|^{\tau p-1}+qC_{\varepsilon}|t|^{q-1}\quad \mbox{a.e.}  \ \xi\in \mathbb{H}^{N}\quad\mbox{and}\quad\mbox{for every}\ t\in\mathbb{R};$$

\item[$(f_{2})^{'}$] there exists $a_{2}>0$, $q_{2}\in(\tau p,Q_\lambda^\ast)$ satisfying
$$F(\xi,t)\geq a_{2}|t|^{q_{2}}\quad \mbox{a.e.}  \ \xi\in \mathbb{H}^{N}\quad\mbox{and}\quad\mbox{for every}\ t\in\mathbb{R};$$

\item[$(f_{3})^{'}$] there exists
$q_{0}\in(\frac{m_{1}\tau p}{m_{0}},Q_\lambda^\ast)$ satisfying
$q_{0}F(\xi,t)\leq f(\xi,t)t$ $\mbox{for every}\
(\xi,t)\in\mathbb{H}^N\times \mathbb{R}$, where $m_0$ and $m_1$ are
the numbers from the condition $(M)$.
\end{itemize}

Poho\u{z}aev \cite{s} was the first to study Kirchhoff equation
problems, and he proved the unique solvability of the mixed problems
of quasi-linear hyperbolic Kirchhoff equations with Dirichlet
boundary conditions. Since then, Kirchhoff type problems have been
receiving increasing attention, especially in various models of
biological and physical systems. More recently, Fiscella and
Valdinoci \cite{F1} have discussed in detail the physical
significance of the fractional Kirchhoff problem and its
application, and proposed a stable Kirchhoff variational problem as
a very realistic model. If the nonlinear term has the convolution
form, many interesting results have been obtained for this kind of
problem. For example, Fan \cite{fan} considered  the following
fractional Choquard-Kirchhoff equation with subcritical or critical
nonlinearity of the form:
\begin{equation}\label{1.11}
\left\{
\begin{array}{lll}
M([u]^{2}_{s})(-\Delta)^{s}u=\lambda\int_{\Omega}\frac{|u|^{p}}{|x-y|^{\mu}}dy|u|^{p-2}u+|u|^{q-2}u
&\mbox{in}\  \Omega, \\
u=0  &\mbox{in}\  \mathbb{R}^N\backslash\Omega,
\end{array} \right.
\end{equation}
where $M(t)=a+b^{\theta-1}$, $\theta\in(1,\frac{2_{s}^{\ast}}{2})$,
$0<s<1$, $\lambda>0$ is a positive parameter, $\Omega$ is a
bounded domain in $\mathbb{R}^N$ with smooth boundary, $0<\mu<N$,
$N>2s$, $\theta<p<2_{\mu,s}^{\ast}=\frac{2N-\mu}{N-2s}$, and
$2\theta<q\leq2_{s}^{\ast}=\frac{2N}{N-2s}$. The existence of
solutions for problem \eqref{1.11} was obtained by using variational methods and
Nehari manifolds.

By using the concentration-compactness lemma and
variational methods, Goel and Sreenadh \cite{go2} proved the existence and
multiplicity of positive solutions of the Choquard-Kirchhoff
equation:
\begin{equation*}
\left\{
\begin{array}{lll}
-M(\|u\|^{2}_{2})\Delta u=\lambda f(x)|u|^{q-2}u+\int_{\Omega}\frac{|u|^{2_{\mu}^{*}}}{|x-y|^{\mu}}dy|u|^{2_{\mu}^{*}-2}u,\quad x\in\Omega,\\
u=0, \quad x\in\partial\Omega,
\end{array} \right.
\end{equation*}
where $M(t)=a+\varepsilon^{p}t^{\theta-1}$,
$2_{\mu}^{\ast}=\frac{2N-\mu}{N-2}$, $f$ is a continuous real valued
sign changing function,
and
$1<q\leq2$.

Liang et. al. \cite{liang2}  considered  the following
Choquard-Kirchhoff equations  with Hardy-Littlewood-Sobolev critical
exponent:
\begin{equation*}
-\left(a + b\int_{\mathbb{R}^N} |\nabla u|^2 dx\right)\Delta u  =
\alpha k(x)|u|^{q-2}u +
\beta\left(\displaystyle\int_{\mathbb{R}^N}\frac{|u(y)|^{2^*_{\mu}}}{|x-y|^{\mu}}dy\right)|u|^{2^*_{\mu}-2}u,
\quad
 x \in \mathbb{R}^N,
\end{equation*}
where $a > 0$, $b \geq 0$,  $0<\mu<4$, $N \geq 3$, $\alpha, \beta$
are real parameters, $2^*_{\mu}=\frac{2N-\mu}{N-2}$ is the critical
exponent in the sense of Hardy-Littlewood-Sobolev inequality, and $k(x)
\in L^r(\mathbb{R}^N)$ with $r = \frac{2^\ast}{2^\ast-q}$. For the
cases $1<q<2$, $q=2$,  $2<q<2^\ast,$ and $4<q<2\cdot2_\mu^\ast$, they
obtained  the existence and multiplicity results by
using the Symmetric Mountain Pass Theorem and genus theory under
suitable conditions.

On the other hand, the study of nonlinear partial differential
equations on the Heisenberg group has brought about widespread
attention of many researchers. At the same time, some authors tried to
establish the existence and multiplicity of solutions for partial
differential equation solutions on the Heisenberg group.  For
example, Liang and Pucci  \cite{Liang5} applied the Symmetric Mountain
Pass Theorem to consider a class of
 the critical
Kirchhoff-Poisson systems on the Heisenberg group. Pucci and
Temperini \cite{pu6} proved the existence of entire nontrivial solutions for
the $(p, q)$ critical systems on the Heisenberg group by an
application of variational methods. Pucci \cite{pu4}  applied the
Mountain Pass Theorem and the Ekeland variational principle to prove
the existence of nontrivial nonnegative solutions of the
Schr\"{o}dinger-Hardy system on the Heisenberg group. However, once
we turn our attention to the critical Choquard equation  on the
Heisenberg group, we immediately notice that the literature is
relatively sparse. We note that Goel and Sreenadh \cite{go} proved
the regularity of solutions and nonexistence of solutions for the
critical Choquard equation on the Heisenberg group by using the
Linking Theorem and the Mountain Pass Theorem.

Sun et al. \cite{sun} studied the following critical
Choquard-Kirchhoff problem on the Heisenberg group:
$$M(\|u\|^2)(-\Delta_{\mathbb{H}}u+u)= \int_{\mathbb{H}^N}\frac{|u(\eta)|^{Q_\lambda^\ast}}{|\eta^{-1}\xi|^\lambda}d\eta|u|^{Q_\lambda^\ast-2}u+\mu f(\xi,u),$$
where $f$ is a Carath\'{e}odory function, $M$ is the Kirchhoff
function, $\Delta_{\mathbb{H}}$ is the Kohn Laplacian on the
Heisenberg group $\mathbb{H}^N$, $\mu>0$ is a parameter, and
$Q_\lambda^\ast=\frac{2Q-\lambda}{Q-2}$ is the critical exponent  in
the sense of Hardy-Littlewood-Sobolev inequality. They were the
first to establish a new version of the concentration-compactness
principle  on the Heisenberg group. Moreover, the existence of
nontrivial solutions were obtained under non-degenerate and
degenerate conditions.  For more fascinating results, see An and Liu
\cite{an}, Bordoni and Pucci \cite{bor}, Liu et al. \cite{Liu2}, Liu
and Zhang \cite{Liu}, Pucci \cite{pu3, pu4}, and Pucci and Temperini
\cite{pu5, pu6}.

Inspired by the above achievements, we prove that problem
\eqref{1.1} has infinitely many solutions for $\mu$ large enough. We
also prove that this equation has $m$ pairs of solutions for $\mu$
small enough and odd nonlinear function $f(x,\cdot)$. In particular,
it should be pointed out that our results are new even in the
Euclidean case.

Before stating the main results of this paper, we present some
notions about the Heisenberg group $\mathbb{H}^N$.
If $\xi=(x,y,t)\in\mathbb{H}^N$, then the definition of this group operation is
$$\ \tau_{\xi}(\xi')=\xi\circ\xi'=(x+x',y+y',t+t'+2(x'y-y'x))\  \  \mbox{for every}  \   \xi,\xi'\in \mathbb{H}^N.$$
Next, $\xi^{-1}=-\xi$ is the inverse and therefore
$(\xi')^{-1}\circ {\xi}^{-1}=(\xi\circ\xi')^{-1} $.

The definition of a natural group of dilations on $\mathbb{H}^N$ is
$\delta_s(\xi)=(sx,sy,s^2t)$, for every $s>0$.
Hence, $\delta_s(\xi_0\circ\xi) =
\delta_s(\xi_0)\circ \delta_s(\xi)$. It can be easily proved that the Jacobian determinant of
dilations $\delta_s: \mathbb{H}^N\to \mathbb{H}^N$ is constant and
equal to $s^Q$ for every $\xi=(x,y,t)\in
\mathbb{H}^N$. The natural number $Q=2N+2$ is called the
homogeneous dimension of $\mathbb{H}^N$ and the critical exponents is
$Q^{*}:= \frac{2Q}{Q-2}$. We define the Kor\'{a}nyi norm as follows
$$|\xi|_{H}=\left[(x^2+y^2)^2+t^2\right]^{\frac{1}{4}} \  \  \mbox{for every}  \   \xi\in \mathbb{H}^N, $$
and we derive this norm from the Heisenberg group's anisotropic dilation.
Hence, the homogeneous degree of the Kor\'{a}nyi norm is equal to 1, in terms of dilations
$$\delta_s:(x,y,t)\mapsto(sx,sy,s^2t) \  \  \mbox{for every}  \   s>0.$$
The set
$$B_{H}(\xi_{0},r)=\{\xi \in \mathbb{H}^{N}:d_{H}(\xi_0,\xi)<r\}, $$
denotes the  Kor\'{a}nyi open ball of radius $r$ centered at $\xi_{0}$.
For the sake of simplicity, we denote $B_{r} = B_{r}(O)$, where $O = (0,0)$ is
the natural origin of $\mathbb{H}^{N}$.

The following vector fields
$$T=\frac{\partial}{\partial t},  \
X_{j}=\frac{\partial}{\partial
x_{j}}+2y_{j}\frac{\partial}{\partial t},  \
Y_{j}=\frac{\partial}{\partial y_{j}}-2x_{j}\frac{\partial}{\partial
t},$$
generate the real Lie algebra of left invariant vector fields for $j=1,\cdots,n$,
which forms a basis satisfying
the Heisenberg regular commutation relation on $\mathbb{H}^{N}$.
This means that
$$[X_{j}, Y_{j}]=-4\delta_{jk}T,  \  [Y_{j},Y_{k}]=[X_{j},X_{k}]=[Y_{j},T]=[X_{j},T]=0.$$
The so-called horizontal vector field is just a vector field with the span of $[X_{j}, Y_{j}]^{n}_{j=1}$.

The Heisenberg gradient on $\mathbb{H}^{N}$ is
$$\nabla_H =(X_{1},X_{2},\cdots,X_{n},Y_{1},Y_{2},\cdots,Y_{n}),$$
and the Kohn Laplacian on $\mathbb{H}^{N}$ is given by
$$\Delta_{H}=\sum_{j=1}^{N}X_{j}^{2}+Y_{j}^{2}=\sum_{j=1}^{N}[\frac{\partial^{2}}{\partial x_{j}^{2}}+\frac{\partial^{2}}{\partial y_{j}^{2}}+4y_{j}\frac{\partial^{2}}{\partial x_{j}\partial t}-4x_{j}\frac{\partial^{2}}{\partial x_{j}\partial t}+4(x_{j}^{2}+y_{j}^{2})\frac{\partial^{2}}{\partial t^{2}}].$$

The Haar measure is invariant under the left translations of the
Heisenberg group and is $Q$-homogeneous in terms of dilations. More
precisely, it is consistent with the $(2n + 1)$-dimensional Lebesgue
measure. Hence, as shown by Leonardi and Masnou \cite{Leonardi},  the
topological dimension $2N + 1$ of $\mathbb{H}^{N}$ is strictly less
than its Hausdorff dimension $Q = 2N + 2$. Next, $|\Omega|$ denotes
the $(2N + 1)$-dimensional Lebesgue measure of any measurable set
$\Omega\subseteq \mathbb{H}^{N}$. Therefore,
$$|\delta_{s}(\Omega)|=s^{Q}|\Omega|,  \   d(\delta_{s}\xi)=s^{Q}d\xi
\
\hbox{and}
\
|B_{H}(\xi_{0},r)|=\alpha_{Q}r^{Q},  \  \mbox{where}  \   \alpha_{Q}=|B_{H}(0,1)|.$$

For the case of critical exponent $\tau=\frac{Q_\lambda^\ast}{p}$, we have the following theorem.

\begin{theorem}\label{the1.1}
Let $\tau=\frac{Q_\lambda^\ast}{p}$, $2<p<\frac{N}{s}$, and suppose
that condition $(M)$ is satisfied. Assume that the nonlinearity
$f(\xi,t)$ is odd in $t$ for fixed $\xi$
 and satisfies conditions $(f_{1})$ and $(f_2)$, and the potential
function $V$ satisfies conditions $(V_1)$ and $(V_2)$. Then problem
\eqref{1.1} has infinitely many solutions  for $\mu$ large enough.
\end{theorem}

For the case of subcritical exponent $\tau\in(1,\frac{Q_\lambda^\ast}{p})$, we also have the following result.

\begin{theorem}\label{the1.2}
Let $\tau\in(1,\frac{Q_\lambda^\ast}{p})$ and suppose that condition
$(M)$ is satisfied. Assume that $f(\cdot,\cdot)$ satisfies
conditions $(f_{1})'$, $(f_{2})',$ and $(f_3)'$, and the potential
function $V$ satisfies conditions $(V_1)$ and $(V_2)$. Then

(i) for every $\mu>0$, there exists $\mu^{\ast}>0$ such that problem
\eqref{1.1} has at least one
nontrivial solution $u_\lambda$ with the
following estimate: for every $\mu\in(0,\mu^{\ast}]$,
\begin{equation}\label{1.2}
\|u_\lambda\|_\mu^{p}\leq\left(\frac{\tau pq_0}{m_0q_0-m_1\tau p}\right)^{\frac{1}{\tau}}\rho^{\frac{1}{\tau}}\mu^{\frac{Q_\lambda^\ast}{Q_\lambda^\ast-\tau p}}
\end{equation}
and
\begin{equation}\label{1.3}
\|u_\lambda\|_{H_{Q_\lambda^\ast}}^{Q_\lambda^\ast}\leq\rho\frac{2q_0Q_\lambda^\ast}{2Q_\lambda^\ast-q_0}\mu^{\frac{\tau
Q_\lambda^\ast}{Q_\lambda^\ast-\tau p}},
\end{equation}
where $\rho=\frac{1}{q_0}(1-\frac{m_1}{m_0})+\frac{1}{\tau p}-\frac{1}{2Q_\lambda^\ast}$.

(ii) if $f(\xi,t)$ is odd with respect to $t$, then for every $m\in
\mathbb{N}$, there exists $\mu_m>0$ such that problem \eqref{1.1}
has at least $m$ pairs of solutions $u_{\lambda,j}$ and
$u_{\lambda,-j}$ $(j=1,2,\cdots,m)$ for $0<\mu<\mu_m$, which satisfy
\eqref{1.2} and \eqref{1.3}.
\end{theorem}

The paper is organized as follows. In Section 2, we shall review some necessary definitions and useful lemmas related to our main proof.
In Section 3, we mainly discuss the critical case $\tau=\frac{Q_\lambda^\ast}{p}$, and give the proof of Theorem \ref{the1.1}.
Finally,  in Section 4 we discuss the subcritical case   and prove Theorem \ref{the1.2}.

\section{Preliminaries}

In this section, we shall review some necessary definitions and
useful lemmas related to our main proof. First, let $u:
\mathbb{H}^N\rightarrow\mathbb{R}$ be a measurable function. We set
$$[u]_{s,p}=\left(\int\int_{\mathbb{H}^{2N}}\frac{|u(\xi)-u(\eta)|^{p}}{|\xi-\eta|^{N+ps}}d\xi
d\eta\right)^{\frac{1}{p}}$$ and define the fractional Sobolev space
$S^{s,p}(\mathbb{H}^N)$ on the Heisenberg group as follows:
$$S^{s,p}(\mathbb{H}^N)=\{u\in L^{p}(\mathbb{H}^N):\ \mbox{$u$ is a measurable function with}\ [u]_{s,p}<\infty\}$$
and the norm
$$\|u\|_{S^{s,p}(\mathbb{H}^N)}=([u]_{s,p}^{p}+|u|_p^{p})^{\frac{1}{p}}\ \mbox{with}\ |u|_p=\left(\int_{\mathbb{H}^N}|u|^pd\xi\right)^{\frac{1}{p}}.$$
Moreover, for $\mu>0$, let $S_\mu$ be the closure of
$C_0^{\infty}(\mathbb{H}^N)$ with respect to the following norm
$$\|u\|_{\mu}=\left(\mu[u]_{s,p}^{p}+\|u\|_{p,V}^{p}\right)^{\frac{1}{p}}\ \mbox{with}\ \|u\|_{p,V}=\left(\int_{\mathbb{H}^N}V(\xi)|u|^pd\xi\right)^{\frac{1}{p}}$$
in the presence of potential $V(\xi)$.

It follows that
 $(S_\mu, \|\cdot\|_{\mu})$ is a
uniformly convex Banach space / this was proved in Pucci et al.
\cite{PXZ1}. Now, we can define the weak solution of problem
\eqref{1.1}.
\begin{definition}\label{def2.1}
We call $u\in S_\mu$ a weak solution of  problem \eqref{1.1} if
\begin{align}\label{2.1}
\begin{split}
&M(\|u\|_{\mu}^{p})(\mu\int\int_{\mathbb{H}^{2N}}\frac{|u(\xi)-u(\eta)|^{p-2}(u(\xi)-u(\eta))}{|\xi-\eta|^{N+ps}}(\varphi(\xi)-\varphi(\eta))d\xi d\eta
+\int_{\mathbb{H}^{N}}V(\xi)|u|^{p-2}u\varphi d\xi)\\
&=\int_{\mathbb{H}^{N}}f(\xi,u)\varphi d\xi
+\int_{\mathbb{H}^{N}}\int_{\mathbb{H}^N}\frac{|u(\eta)|^{Q_{\lambda}^{\ast}}}{|\eta^{-1}\xi|^{\lambda}}d\eta|u(\xi)|^{Q_{\lambda}^{\ast}-2}
u(\xi)\varphi(\xi)d\xi\quad\mbox{for every}\ \varphi\in S_\mu.\\
\end{split}
\end{align}
\end{definition}

The corresponding energy functional $I_{\mu}(u): S_\mu\rightarrow
\mathbb{R}$ of problem \eqref{1.1} is
\begin{equation}\label{2.2}
I_{\mu}(u)=\frac{1}{p}\widetilde{M}(\|u\|_{\mu}^{p})-\frac{1}{2Q_{\lambda}^{\ast}}\int_{\mathbb{H}^{N}}\int_{\mathbb{H}^N}
\frac{|u(\xi)|^{Q_{\lambda}^{\ast}}|u(\eta)|^{Q_{\lambda}^{\ast}}}{|\eta^{-1}\xi|^{\lambda}}d\eta d\xi-\int_{\mathbb{H}^{N}}F(\xi,u)d\xi,
\end{equation}
where $\widetilde{M}(t)=\int_{0}^{t}M(s)ds$. It is easy to   prove
that $I_{\mu}\in C^{1}(S_\mu, R)$ and its critical points are
solutions of problem \eqref{1.1}.

Next, we define
\begin{equation}\label{2.3}
H_{Q_{\lambda}^{\ast}}=\inf_{u\in S_\mu\setminus\{0\}}\frac{[u]_{s,p}^{p}}{\left(\int_{\mathbb{H}^{N}}\int_{\mathbb{H}^N}
\frac{|u(\xi)|^{Q_{\lambda}^{\ast}}|u(\eta)|^{Q_{\lambda}^{\ast}}}{|\eta^{-1}\xi|^{\lambda}}d\eta d\xi\right)^{\frac{p}{Q_{\lambda}^{\ast}}}}
\end{equation}
and
\begin{equation}\label{2.7}
\|u\|_{H_{Q_{\lambda}^{\ast}}}^{Q_{\lambda}^{\ast}}=\int_{\mathbb{H}^{N}}\int_{\mathbb{H}^N}
\frac{|u(\xi)|^{Q_{\lambda}^{\ast}}|u(\eta)|^{Q_{\lambda}^{\ast}}}{|\eta^{-1}\xi|^{\lambda}}d\eta d\xi.
\end{equation}
By \eqref{2.3}, we know that $H_{Q_{\lambda}^{\ast}}$ is positive.

Let $S$ denote the completion of $C_0^{\infty}(\mathbb{H}^N)$ with respect to the norm
$$\|u\|_{S}=\left([u]_{s,p}^{p}+\|u\|_{p,V}^{p}\right)^{\frac{1}{p}}\ \mbox{with}\ \|u\|_{p,V}=\left(\int_{\mathbb{H}^N}V(\xi)|u|^pd\xi\right)^{\frac{1}{p}}.$$

Note that for every fixed $\mu>0$, the norm $\|u\|_{W}$ is
equivalent to $\|u\|_{\mu}$. Invoking Bordoni and Pucci \cite{bor},
Pucci \cite{pu4}, we can get the following embedding result.
\begin{lemma}\label{lem2.1}
Let $V(\xi)$ satisfy condition $(V_1)$. Then for every $\gamma\in [p,Q_{\lambda}^{\ast}]$, the embedding
\begin{equation}\label{2.4}
S_\mu\hookrightarrow S^{s,p}(\mathbb{H}^N)\hookrightarrow L^{\gamma}(\mathbb{H}^N)
\end{equation}
is continuous. Moreover, for every $\gamma\in [p,Q_{\lambda}^{\ast})$, the embedding $S_\mu\hookrightarrow
L^{\gamma}(\mathbb{H}^N)$ is compact. In addition, there is a constant $C_\gamma>0$ satisfying
$$|u|_\gamma\leq C_\gamma\|u\|_{\mu}\quad\mbox{for every}\ u\in S_\mu.$$
\end{lemma}

\begin{lemma}\label{lem2.2}
Let $V$ satisfy conditions $(V_1)$ and $(V_2)$, and let $\gamma\in
[p,Q_{\lambda}^{\ast})$ be a fixed exponent. Then for every bounded
sequence $\{u_n\}_n$ in $S_\mu$, which up to a subsequence and $u\in
S_\mu$  satisfies
$$u_n\rightarrow u \quad\mbox{in}\ L^{\gamma}(\mathbb{H}^N)\quad\mbox{as}\ n\rightarrow\infty.$$
\end{lemma}

Next, let $D^{s,p}(\mathbb{H}^N)$ be the completion of
$C_0^{\infty}(\mathbb{H}^N)$ with respect to the Gagliardo semi-norm
$[\cdot]_{s,p}$. Si\-mi\-larly to the proof o Sun et al. \cite[Theorem
3.1]{sun}, we get the following lemma.
\begin{lemma}\label{lem2.4}
For every $0\leq sp$, let $\{u_n\}_n\subset D^{s,p}(\mathbb{H}^N)$
be a bounded sequence satisfying
\begin{equation*}
\begin{cases}
u_{n}\rightharpoonup u,\\
\int_{\mathbb{H}^{N}}\frac{|u_n(\xi)-u_n(\eta)|^{p}}{|\xi-\eta|^{N+ps}}d\eta\rightharpoonup\kappa\geq\int_{\mathbb{H}^{N}}\frac{|u(\xi)-u(\eta)|^{p}}
{|\xi-\eta|^{N+ps}}d\eta+\sum_{j\in J}\kappa_{j}\delta_{x_{j}},\\
\int_{\mathbb{H}^N}\frac{|u_{n}(\eta)|^{Q_{\lambda}^{\ast}}}{|\eta^{-1}\xi|^{\lambda}}d\eta|u_{n}(\xi)|^{Q_{\lambda}^{\ast}}\rightharpoonup v=
\int_{\mathbb{H}^N}\frac{|u(\eta)|^{Q_{\lambda}^{\ast}}}{|\eta^{-1}\xi|^{\lambda}}d\eta|u(\xi)|^{Q_{\lambda}^{\ast}}+\sum_{j\in J}v_{j}\delta_{x_{j}},
\end{cases}
\end{equation*}
where J is an at most countable index set, $x_{j}\in\mathbb{H}^N,$ and $\delta_{x_{j}}$ is the Dirac mass at $x_{j}$.
Furthermore, let
\begin{equation*}
\kappa_\infty=\lim_{R\rightarrow\infty}\limsup_{n\rightarrow\infty}\int_{\{\xi\in\mathbb{H}^N:|\xi|>R\}}\int_{\mathbb{H}^{N}}
\frac{|u_n(\xi)-u_n(\eta)|^{p}}{|\xi-\eta|^{N+ps}}d\eta d\xi,
\end{equation*}

\begin{equation*}
v_\infty=\lim_{R\rightarrow\infty}\limsup_{n\rightarrow\infty}\int_{\{\xi\in\mathbb{H}^N:|\xi|>R\}}
\int_{\mathbb{H}^N}\frac{|u_{n}(\eta)|^{Q_{\lambda}^{\ast}}}{|\eta^{-1}\xi|^{\lambda}}d\eta|u_{n}(\xi)|^{Q_{\lambda}^{\ast}}d\xi.
\end{equation*}

Then for the energy at infinity, the following holds:
\begin{equation}\label{2.9}
\limsup_{n\rightarrow\infty}\int_{\mathbb{H}^{N}}\int_{\mathbb{H}^{N}}
\frac{|u_n(\xi)-u_n(\eta)|^{p}}{|\xi-\eta|^{N+ps}}d\eta
d\xi=\int_{\mathbb{H}^{N}}d\kappa+\kappa_\infty
\end{equation}
and
\begin{equation}\label{2.10}
\limsup_{n\rightarrow\infty}\int_{\mathbb{H}^N}
\int_{\mathbb{H}^N}\frac{|u_{n}(\eta)|^{Q_{\lambda}^{\ast}}|u_{n}(\xi)|^{Q_{\lambda}^{\ast}}}{|\eta^{-1}\xi|^{\lambda}}d\eta d\xi=\int_{\mathbb{H}^N}dv+v_{\infty}.
\end{equation}

In addition,
\begin{equation*}
\kappa_j\geq H_{Q_{\lambda}^{\ast}}v_j^{\frac{p}{Q_{\lambda}^{\ast}}}
\end{equation*}
and
\begin{equation*}
\kappa_\infty\geq H_{Q_{\lambda}^{\ast}}v_\infty^{\frac{p}{Q_{\lambda}^{\ast}}}.
\end{equation*}
\end{lemma}

\section{Proof of Theorem 1.1}

For the case of the critical exponent
$\tau=\frac{Q_\lambda^\ast}{p}$ and the Kirchhoff function
$M(\cdot)$ satisfying condition $(M)$, we use this section to prove
the existence of an infinite number of solutions to problem
\eqref{1.1}.

\begin{lemma}\label{lem3.1}
Let $\tau=\frac{Q_\lambda^\ast}{p}$, $2<p<\frac{N}{s},$ and let
condition $(M)$ be satisfied. Suppose that the nonlinearity
$f(\xi,t)$ is odd in t for fixed $\xi$,  $f(\cdot,\cdot)$ satisfies
conditions $(f_{1})$ and $(f_{2})$, and the potential function $V$
satisfies conditions $(V_1)$ and $(V_2)$. Then $I_\mu$ is bounded
from below for
$\mu>\frac{2^{p}H_{Q_{\lambda}^{\ast}}^{-Q_{\lambda}^{\ast}/p}}{m_0}$
and $I_\mu$ is even.
\end{lemma}
\begin{proof}
By $(f_1)$,  there exists $C_{0}>0$ satisfying
$$|f(\xi,t)|\leq p|t|^{p-1}+C_{0}q|t|^{q-1}$$
and
$$F(\xi,t)\leq |t|^{p}+C_{0}|t|^{q}\quad \mbox{for a.e.}  \ \xi\in \mathbb{H}^{N}\ \mbox{and all }\ t\in\mathbb{R}.$$
For every $u\in S_{\mu}$, by condition $(M)$, we have
\begin{equation*}
I_{\mu}(u)\geq\frac{m_0}{p\tau}\|u\|_{\mu}^{p\tau}-\frac{1}{2Q_{\lambda}^{\ast}}\int_{\mathbb{H}^{N}}\int_{\mathbb{H}^N}
\frac{|u(\eta)|^{Q_{\lambda}^{\ast}}|u(\xi)|^{Q_{\lambda}^{\ast}}}{|\eta^{-1}\xi|^{\lambda}}d\eta d\xi
-\int_{\mathbb{H}^{N}}|u|^{p}dx-C_0\int_{\mathbb{H}^{N}}|u|^{q}dx.
\end{equation*}
By the definition of $H_{Q_{\lambda}^{\ast}}$ and Lemma
\ref{lem2.1}, for $\tau=\frac{Q_\lambda^\ast}{p}$, we have
\begin{align*}
\begin{split}
I_{\mu}(u)&\geq\frac{m_0}{p\tau}\|u\|_{\mu}^{p\tau}-\frac{1}{2Q_{\lambda}^{\ast}}\int_{\mathbb{H}^{N}}\int_{\mathbb{H}^N}
\frac{|u(\eta)|^{Q_{\lambda}^{\ast}}|u(\xi)|^{Q_{\lambda}^{\ast}}}{|\eta^{-1}\xi|^{\lambda}}d\eta d\xi-C_1\|u\|_{\mu}^{p}-C_1\|u\|_{\mu}^{q}\\
&\geq\left(\frac{m_0}{Q_\lambda^\ast}-\frac{\mu^{-1}}{2Q_{\lambda}^{\ast}}H_{Q_{\lambda}^{\ast}}^{-Q_{\lambda}^{\ast}/p}\right)\|u\|_{\mu}^{Q_\lambda^\ast}
-C_1\|u\|_{\mu}^{p}-C_1\|u\|_{\mu}^{q}.\\
\end{split}
\end{align*}
Since
$\mu>\frac{2^{p}H_{Q_{\lambda}^{\ast}}^{-Q_{\lambda}^{\ast}/p}}{m_0}>\frac{H_{Q_{\lambda}^{\ast}}^{-Q_{\lambda}^{\ast}/p}}{2m_0}$
and $2\leq p<q<Q_{\lambda}^{\ast}$, we can deduce that
$\frac{m_0}{Q_\lambda^\ast}-\frac{\mu^{-1}}{2Q_{\lambda}^{\ast}}H_{Q_{\lambda}^{\ast}}^{-Q_{\lambda}^{\ast}/p}=M_1>0.$
There exists a small constant $\varepsilon_1$ such that
$\varepsilon_1C_2<M_1$. By Young's inequality, we can deduce that
$$I_{\mu}(u)\geq(M_1-\varepsilon_1C_2)\|u\|_{\mu}^{Q_{\lambda}^{\ast}}-C_3.$$
Thus, we get $I_{\mu}(u)\geq-C_3$. Moreover, since $f(\xi,t)$ is odd
in $t$ for fixed $\xi$, we obtain that $I_{\mu}$ is even. The proof
of Lemma \ref{lem3.1} is complete.
\end{proof}

\begin{lemma}\label{lem3.2}
Under the assumptions of Lemma \ref{lem3.1},
$I_{\mu}$ satisfies $(PS)_c$ condition for every
$\mu>\frac{2^{p}H_{Q_{\lambda}^{\ast}}^{-Q_{\lambda}^{\ast}/p}}{m_0}$.
\end{lemma}

\begin{proof}
Take $\{u_n\}_n\subset S_\mu$ to be a
 $(PS)_c$ sequence of the functional $I_{\mu}$, that is,
$$I_{\mu}(u_n)\rightarrow c,\quad I_{\mu}'(u_n)\rightarrow0\quad\mbox{as}\ n\rightarrow\infty.$$
Then we claim that $\{u_n\}_n$ is bounded in $S_\mu$. In fact, from
the proof of Lemma \ref{lem3.1}, we can conclude that
$$c+o(1)\geq I_{\mu}(u_n)\geq(M_1-\varepsilon_1C_2)\|u_n\|_{\mu}^{Q_{\lambda}^{\ast}}-C_3.$$
Note that $M_1-\varepsilon_1C_2>0$ when $\varepsilon_1$ small
enough, so  $\{u_n\}_n$ is uniformly bounded in
$S_\mu$. This means that there is a subsequence of $\{u_n\}_n$ and
$u\in S_\mu$ satisfying
\begin{align}\label{3.1}
\begin{split}
&u_n\rightharpoonup u\quad \mbox{in}\ S_\mu\ \mbox{and in}\ L^{Q_{\lambda}^{\ast}}(\mathbb{H}^N),\\
&u_n\rightarrow u\quad \mbox{a.e. in}\ \mathbb{H}^N,\\
&|u_{n}|^{Q_{\lambda}^{\ast}-2}u_n\rightharpoonup|u|^{Q_{\lambda}^{\ast}-2}u\quad \mbox{in}\  L^{\frac{Q_{\lambda}^{\ast}}{Q_{\lambda}^{\ast}-1}}(\mathbb{H}^N)
\end{split}
\end{align}
as $n\rightarrow\infty$.
By the Br\'{e}zis-Lieb type inequality, we can obtain
\begin{equation}\label{2.8}
\|u_n-u\|_{H_{Q_{\lambda}^{\ast}}}^{Q_{\lambda}^{\ast}}=\|u_n\|_{H_{Q_{\lambda}^{\ast}}}^{Q_{\lambda}^{\ast}}
-\|u\|_{H_{Q_{\lambda}^{\ast}}}^{Q_{\lambda}^{\ast}}+o(1).
\end{equation}
Therefore, we have
\begin{align}\label{3.2}
\begin{split}
&\lim_{n\rightarrow\infty}\int_{\mathbb{H}^{N}}\int_{\mathbb{H}^N}
\frac{(|u_n(\xi)|^{Q_{\lambda}^{\ast}}|u_n(\eta)|^{Q_{\lambda}^{\ast}-2}u_n(\eta)-|u(\xi)|^{Q_{\lambda}^{\ast}}|u(\eta)|^{Q_{\lambda}^{\ast}-2}u(\eta))
(u_n(\eta)-u(\eta))}{|\eta^{-1}\xi|^{\lambda}}d\eta d\xi\\
&=\int_{\mathbb{H}^{N}}\int_{\mathbb{H}^N}
\frac{|u_n(\xi)-u(\xi)|^{Q_{\lambda}^{\ast}}|u_n(\eta)-u(\eta)|^{Q_{\lambda}^{\ast}}}{|\eta^{-1}\xi|^{\lambda}}d\eta d\xi+o(1).
\end{split}
\end{align}
Now, we shall demonstrate that
\begin{equation}\label{3.3}
\lim_{n\rightarrow\infty}\int_{\mathbb{H}^{N}}(f(\xi,u_n)-f(\xi,u))(u_n-u)d\xi=0.
\end{equation}
Up to a subsequence, it follows by Lemma \ref{lem2.2} that
$u_n\rightarrow u$ in $L^{\gamma}(\mathbb{H}^{N})$ for
$\gamma=p,q\in[p,Q_{\lambda}^{\ast})$. By $(f_1)$, we get
$$|f(\xi,t)|\leq |t|^{p-1}+C_{f}|t|^{q-1}\quad \mbox{a.e.}  \ \xi\in \mathbb{H}^{N}\quad\mbox{and}\quad\mbox{for every}\ t\in\mathbb{R}.$$
By the H\"{o}lder inequality, we get that
\begin{align*}
\begin{split}
\left|\int_{\mathbb{H}^{N}}(f(\xi,u_n)-f(\xi,u))(u_n-u)d\xi\right|&\leq\int_{\mathbb{H}^{N}}\left((|u_n|^{p-1}+|u|^{p-1})|u_n-u|+C_f
(|u_n|^{q-1}+|u|^{q-1})|u_n-u|\right)d\xi\\
&\leq(|u_n|_p^{p-1}+|u|_p^{p-1})|u_n-u|_p+C_f(|u_n|_q^{q-1}+|u|_q^{q-1})|u_n-u|_q,\\
\end{split}
\end{align*}
which implies that \eqref{3.3}  holds.

Next, for every fixed $u\in S_\mu$, we define the following linear
function $L(u)$ on $S_\mu$:
\begin{equation*}
\langle
L(u),\varphi\rangle=\mu\int\int_{\mathbb{H}^{2N}}\frac{|u(\xi)-u(\eta)|^{p-2}(u(\xi)-u(\eta))}{|\eta^{-1}\xi|^{N+ps}}(\varphi(\xi)-\varphi(\eta))d\xi
d\eta+\int_{\mathbb{H}^{N}}V(\xi)|u|^{p-2}u\varphi d\xi
\end{equation*}
for every $\varphi\in S_\mu$.

Next, we show that
the linear function $L(u)$ is bounded. In fact, it follows by the H\"{o}lder inequality that
\begin{align*}
\begin{split}
|\langle L(u),\varphi\rangle|&\leq\mu[u]_{s,p}^{p-1}[v]_{s,p}+\left(\int_{\mathbb{H}^{N}}V(\xi)|u|^{p}d\xi\right)^{\frac{p-1}{p}}
\left(\int_{\mathbb{H}^{N}}V(\xi)|v|^{p}d\xi\right)^{\frac{1}{p}}\\
&\leq\left([u]_{s,p}^{p-1}+\left(\int_{\mathbb{H}^{N}}V(\xi)|u|^{p}d\xi\right)^{\frac{p-1}{p}}\right)\|\varphi\|_\mu
\end{split}
\end{align*}
and the following equality \eqref{3.4} holds due to
$u_n\rightharpoonup u$ in $S_\mu$,
\begin{equation}\label{3.4}
\lim_{n\rightarrow\infty}\langle L(u),u_n-u\rangle=0.
\end{equation}
We begin to prove that $\|u_n-u\|_{\mu}\rightarrow0$ as
$n\rightarrow\infty$. Let us assume that in general,
$\lim_{n\rightarrow\infty}\|u_n-u\|_{\mu}=d\neq0$. Since $\{u_n\}_n$
is a $(PS)_c$ sequence, by \eqref{3.2}, \eqref{3.3} and
\eqref{3.4},
we obtain that
\begin{align}\label{3.5}
\begin{split}
o(1)=&\langle I_{\mu}'(u_n),u_n-u\rangle-\langle I_{\mu}'(u),u_n-u\rangle\\
=&M(\|u_n\|_{\mu}^{p})\langle L(u_n),u_n-u\rangle-M(\|u\|_{\mu}^{p})\langle L(u),u_n-u\rangle\\
&
-\int_{\mathbb{H}^{N}}\int_{\mathbb{H}^N}\frac{(|u_n(\xi)|^{Q_{\lambda}^{\ast}}|u_n(\eta)|^{Q_{\lambda}^{\ast}-2}u_n(\eta)-|u(\xi)|^{Q_{\lambda}^{\ast}}
|u(\eta)|^{Q_{\lambda}^{\ast}-2}u(\eta))(u_n(\eta)-u(\eta))}{|\eta^{-1}\xi|^{\lambda}}d\eta d\xi\\
=&M(\|u_n\|_{\mu}^{p})\langle
L(u_n)-L(u),u_n-u\rangle-\int_{\mathbb{H}^{N}}\int_{\mathbb{H}^N}
\frac{|u_n(\xi)-u(\xi)|^{Q_{\lambda}^{\ast}}|u_n(\eta)-u(\eta)|^{Q_{\lambda}^{\ast}}}{|\eta^{-1}\xi|^{\lambda}}d\eta
d\xi.
\end{split}
\end{align}
Let us consider each term on the right hand side of the above
formula separately. By the Br\'{e}zis-Lieb type inequality, we get
\begin{equation}\label{2.5}
[u_n-u]_{s,p}^{p}=[u_n]_{s,p}^{p}-[u]_{s,p}^{p}+o(1)
\end{equation}
and
\begin{equation}\label{2.6}
|V(\xi)^{\frac{1}{p}}(u_n-u)|_p^{p}=|V(\xi)^{\frac{1}{p}}u_n|_p^{p}-|V(\xi)^{\frac{1}{p}}u|_p^{p}+o(1).
\end{equation}
Thus, for $M(\|u_n\|_{\mu}^{p})$,  we get
\begin{equation}\label{3.6}
M(\|u_n\|_{\mu}^{p})=M(\|u_n-u\|_{\mu}^{p}+\|u\|_{\mu}^{p})+o(1).
\end{equation}
For $\langle L(u_n)-L(u),u_n-u\rangle$, we apply the following inequality (see Kichenassamy and Veron \cite{sk}),
\begin{equation}\label{3.7}
|\xi-\eta|^{p}\leq2^{p}(|\xi|^{p-2}\xi-|\eta|^{p-2}\eta)(\xi-\eta)\quad\mbox{for}\ p\geq2
\end{equation}
for every $\xi,\eta\in\mathbb{H^{N}}$.

Next, we put \eqref{3.6} and \eqref{3.7} into \eqref{3.5} to 
get
 the
following estimate:
\begin{align}\label{3.8}
\begin{split}
o(1)+\int_{\mathbb{H}^{N}}\int_{\mathbb{H}^N}
\frac{|u_n(\xi)-u(\xi)|^{Q_{\lambda}^{\ast}}|u_n(\eta)-u(\eta)|^{Q_{\lambda}^{\ast}}}{|\eta^{-1}\xi|^{\lambda}}d\eta d\xi&=M(\|u_n-u\|_{\mu}^{p}+\|u\|_{\mu}^{p})\langle L(u_n)-L(u),u_n-u\rangle\\
&\geq m_0(\|u_n-u\|_{\mu}^{p}+\|u\|_{\mu}^{p})^{\tau-1}\frac{\|u_n-u\|_{\mu}^{p}}{2^{p}}\\
&\geq\frac{m_0}{2^{p}}\|u_n-u\|_{\mu}^{\tau
p}=\frac{m_0}{2^{p}}\|u_n-u\|_{\mu}^{Q_{\lambda}^{\ast}}.
\end{split}
\end{align}
For convolution term $\int_{\mathbb{H}^{N}}\int_{\mathbb{H}^N}
\frac{|u_n(\xi)-u(\xi)|^{Q_{\lambda}^{\ast}}|u_n(\eta)-u(\eta)|^{Q_{\lambda}^{\ast}}}{|\eta^{-1}\xi|^{\lambda}}d\eta
d\xi$, by \eqref{2.3}, one has
\begin{align}\label{3.9}
\begin{split}
\int_{\mathbb{H}^{N}}\int_{\mathbb{H}^N}
\frac{|u_n(\xi)-u(\xi)|^{Q_{\lambda}^{\ast}}|u_n(\eta)-u(\eta)|^{Q_{\lambda}^{\ast}}}{|\eta^{-1}\xi|^{\lambda}}d\eta d\xi\leq H_{Q_{\lambda}^{\ast}}^{-Q_{\lambda}^{\ast}/p}
[u_n-u]_{s,p}^{Q_{\lambda}^{\ast}}\leq \mu^{-1}H_{Q_{\lambda}^{\ast}}^{-Q_{\lambda}^{\ast}/p}\|u_n-u\|_{\mu}^{Q_{\lambda}^{\ast}}.
\end{split}
\end{align}
Finally, we put \eqref{3.9} into \eqref{3.8}, and get
\begin{equation*}
o(1)+\mu^{-1}H_{Q_{\lambda}^{\ast}}^{-Q_{\lambda}^{\ast}/p}\|u_n-u\|_{\mu}^{Q_{\lambda}^{\ast}}\geq\frac{m_0}{2^{p}}\|u_n-u\|_{\mu}^{Q_{\lambda}^{\ast}}.
\end{equation*}
Letting $n\rightarrow\infty$, one has
\begin{equation}\label{3.10}
\mu^{-1}2^{p}H_{Q_{\lambda}^{\ast}}^{-Q_{\lambda}^{\ast}/p}d^{Q_{\lambda}^{\ast}}\geq m_0d^{Q_{\lambda}^{\ast}},
\end{equation}
which implies that $d=0$ or
\begin{equation}\label{3.11}
\mu\leq\frac{2^{p}}{m_0}H_{Q_{\lambda}^{\ast}}^{-Q_{\lambda}^{\ast}/p}.
\end{equation}
This contradicts the condition
$\mu>\frac{2^{p}}{m_0}H_{Q_{\lambda}^{\ast}}^{-Q_{\lambda}^{\ast}/p}$,
which implies that $d=0$. Thus, $u_n\rightarrow u$ in $S_\mu$ when
$\mu>\frac{2^{p}}{m_0}H_{Q_{\lambda}^{\ast}}^{-Q_{\lambda}^{\ast}/p}$.
This completes the proof of Lemma \ref{lem3.2}.
\end{proof}

In order to prove Theorem \ref{the1.1} for problem \eqref{1.1} under
critical conditions, we first review some basic results on the
Krasnoselskii genus (see Clark \cite{Clark}, Rabinowitz
\cite{Rabinowitz}). Let $Y$ be a Banach space and $Z_2=\{id,-id\}$
the symmetric group. We set
$$Z=\{X\subset Y\setminus\{0\}: X\ \mbox{is closed and}\ X=-X\}.$$
\begin{definition}\label{def3.1}
For any $X\in Z$, we define the Krasnoselskii genus of $X$ as follows:
$$\gamma(X)=\inf\{m: \mbox{there exists}\ h\in(C,\mathbb{H}^{m}\setminus\{0\})\ \mbox{and}\ h\ \mbox{is odd}\}.$$
\end{definition}
We define $\gamma(X)=\infty$ if such $k$ does not exist,
and we set $\gamma(\emptyset)=0$.

\begin{lemma}\label{lem3.4}
Let $Y=\mathbb{H}^{N}$ and denote the boundary of $\Omega\in\mathbb{H}^{N}$ by $\partial\Omega,$ which is a symmetric bounded open subset.
Then $\gamma(\partial\Omega)=N$.
\end{lemma}
We denote the unit sphere in $\mathbb{H}^{N}$ by $\mathbb{S}^{N-1}$. We deduce by Lemma \ref{lem3.4} that $\gamma(\mathbb{S}^{N-1})=N$.
The following result helps to prove the existence of an infinite number of solutions to problem \eqref{1.1}.
\begin{lemma}\label{lem3.3} (see Clark \cite{Clark})
Let $I\in C^{1}(Y,\mathbb{R})$ satisfy the Palais-Smale condition.
In addition, we assume that:\\
(i) $I$ is bounded from below and even;\\
(ii) there exists a compact set $K\in Z$ satisfying $\gamma(K)=k$ and $\sup_{x\in K}I(u)<I(0)$.\\
Then the critical value of $I$ is less than $I(0)$
and $I$ has at least $k$ pairs of distinct critical points.
\end{lemma}

\begin{proof}[Proof of Theorem \ref{the1.1}]
Let $e_1, e_2, \cdots,$ be a basis for $S_\mu$. For each
$k\in \mathbb{N}$, $k$ vectors $e_1, e_2, \cdots, e_k$ generate
$Y_k=\mbox{span}\{e_1, e_2, \cdots, e_k\}$, which is a subspace of $S_\mu$.
Since, $p<q_1<Q_{\lambda}^{\ast}$, we deduce that
$Y_k\hookrightarrow L^{q_1}(\mathbb{H}^{N})$. Considering that all
norms of finite-dimensional Banach spaces are equivalent, there
is a positive $C(k)$ that depends only on $k$ and satisfies
\begin{equation}\label{3.12}
\|u\|_{\mu}^{q_1}\leq C(k)\int_{\mathbb{H}^{N}}|u|^{q_1}d\xi\quad\mbox{for every}\ u\in Y_k.
\end{equation}
Under $\tau=\frac{Q_\lambda^\ast}{p}$, for every
$u\in Y_k$, by conditions $(M)$ and $(f_2)$, we can deduce
that
\begin{align*}
\begin{split}
I_{\mu}(u)&\leq\frac{m_1}{p\tau}\|u\|_{\mu}^{p\tau}-\frac{1}{2Q_{\lambda}^{\ast}}\int_{\mathbb{H}^{N}}\int_{\mathbb{H}^N}
\frac{|u(\eta)|^{Q_{\lambda}^{\ast}}|u(\xi)|^{Q_{\lambda}^{\ast}}}{|\eta^{-1}\xi|^{\lambda}}d\eta d\xi-a_1C(k)\|u\|_{\mu}^{q_1}\\
&\leq\left(\frac{m_1}{Q_\lambda^\ast}\|u\|_{\mu}^{Q_\lambda^\ast-q_1}-a_1C(k)\right)\|u\|_{\mu}^{q_1}.\\
\end{split}
\end{align*}
Taking a sufficiently small constant $R>0$ satisfying
$$\frac{m_1}{Q_\lambda^\ast}R^{Q_\lambda^\ast-q_1}<a_1C(k),$$
we get for every $r\in(0,R)$ and $u\in\Lambda=\{u\in Y_k:\|u\|_{\mu}=r\}$,
the following equality \eqref{3.13}
\begin{equation}\label{3.13}
I_{\mu}(u)\leq r^{q_1}\left(\frac{m_1}{Q_\lambda^\ast}r^{Q_\lambda^\ast-q_1}-a_1C(k)\right)\leq R^{q_1}\left(\frac{m_1}{Q_\lambda^\ast}R^{Q_\lambda^\ast-q_1}-a_1C(k)\right)<0=I_{\mu}(0),
\end{equation}
which implies that
$$\sup_{u\in\Lambda}I_{\mu}(u)<0.$$
Note that $\Lambda$ is a homeomorphism of $\mathbb{S}^{k-1}$, so $\gamma(\Lambda)=k$ by Lemma \ref{lem3.4}.
It can be deduced from Lemma \ref{lem3.3} that $I_{\mu}$ has at least $k$ pairs of distinct critical points.
Since $k$ is arbitrary, we have an infinite number of pairs of distinct critical points for $I_{\mu}$ in $S_\mu$.
The proof of Theorem \ref{the1.1} is complete.
\end{proof}

\section{Proof of Theorem 1.2}

For the case of critical exponent
$\tau\in(1,\frac{Q_\lambda^\ast}{p})$ and the Kirchhoff function
$M(\cdot)$ satisfying condition $(M)$, we shall prove in this
section the existence and multiplicity of solutions to problem
\eqref{1.1}.
\begin{lemma}\label{lem4.1}
Let $\tau\in(1,\frac{Q_\lambda^\ast}{p})$ and suppose that the
Kirchhoff function $M(\cdot)$ satisfies condition $(M)$. Assume that
the nonlinearity $f(\cdot,\cdot)$ satisfies condition $(f_3)'$. If
$\{u_n\}_n$ is a $(PS)_c$ sequence of the functional $I_{\mu}$, then
$\{u_n\}_n$ is bounded in $S_\mu$.
\end{lemma}
\begin{proof}
Let $\{u_n\}_n$ be a $(PS)_c$ sequence. By conditions $(M)$ and
$(f_3)'$, we have
\begin{align*}
\begin{split}
c+o(1)(1+\|u_n\|_{\mu})&\geq I_{\mu}(u_n)-\frac{1}{q_0}\langle I_{\mu}'(u_n),u_n\rangle\\
&\geq\left(\frac{m_0}{p\tau}-\frac{m_1}{q_0}\right)\|u\|_{\mu}^{p\tau}+\left(\frac{1}{q_0}-\frac{1}{2Q_{\lambda}^{\ast}}\right)\int_{\mathbb{H}^{N}}\int_{\mathbb{H}^N}
\frac{|u(\eta)|^{Q_{\lambda}^{\ast}}|u(\xi)|^{Q_{\lambda}^{\ast}}}{|\eta^{-1}\xi|^{\lambda}}d\eta d\xi\\
&\quad+\int_{\mathbb{H}^{N}}\left(\frac{f(\xi,u_n)u_n}{q_0}-F(\xi,u_n)\right)d\xi
\geq\left(\frac{m_0}{p\tau}-\frac{m_1}{q_0}\right)\|u_n\|_{\mu}^{p\tau}.\\
\end{split}
\end{align*}
Since $\left(\frac{m_0}{p\tau}-\frac{m_1}{q_0}\right)>0$, we obtain
that $\{u_n\}_n$ is bounded in $S_\mu$. This completes the proof of
Lemma \ref{lem4.1}.
\end{proof}

\begin{lemma}\label{lem4.2}
Let $\tau\in(1,\frac{Q_\lambda^\ast}{p})$ and suppose that the
Kirchhoff function $M(\cdot)$ satisfies condition $(M)$. Assume that
$f(\cdot,\cdot)$ satisfies conditions $(f_{1})'$-$(f_{3})'$, and the
potential function $V$ satisfies conditions $(V_1)$ and $(V_2)$.
Then for every $\mu>0$ and all $c\in(0,
\rho(m_0\mu^{\tau}H_{Q_{\lambda}^{\ast}}^{\tau})^{\frac{Q_{\lambda}^{\ast}}{Q_{\lambda}^{\ast}-\tau
p}})$, $I_\mu$ satisfies the $(PS)_c$ condition, where
$\rho=\frac{1}{\tau
p}-\frac{1}{2Q_{\lambda}^{\ast}}+\frac{1}{q_0}(1-\frac{m_1}{m_0})$.
\end{lemma}

\begin{proof}
We need to divide the proof into two cases, due to the degenerate nature of problem \eqref{1.1}: either $\inf_{n\in N}\|u_n\|_{\mu}=l>0$ or $\inf_{n\in N}\|u_n\|_{\mu}=l=0$.

Case I: $\inf_{n\in N}\|u_n\|_{\mu}=l>0$. Since $\{u_n\}_n$
is
a $(PS)_c$ sequence,
we can
deduce
from
 Lemma \ref{lem4.1} that $\{u_n\}_n$ is bounded in $S_\mu$.
Next, by Lemma \ref{lem2.4}, up to a subsequence,
there is a non-negative function $u\in S_\mu$ satisfying $u_n\rightharpoonup u$ in $S_\mu$,
\begin{equation}\label{4.1}
\int_{\mathbb{H}^{N}}\frac{|u_n(\xi)-u_n(\eta)|^{p}}{|\xi-\eta|^{N+ps}}d\eta\rightharpoonup\kappa\geq\int_{\mathbb{H}^{N}}\frac{|u(\xi)-u(\eta)|^{p}}
{|\xi-\eta|^{N+ps}}d\eta+\sum_{j\in J}\kappa_{j}\delta_{x_{j}}
\end{equation}
and
\begin{equation}\label{4.2}
\int_{\mathbb{H}^N}\frac{|u_{n}(\eta)|^{Q_{\lambda}^{\ast}}}{|\eta^{-1}\xi|^{\lambda}}d\eta|u_{n}(\xi)|^{Q_{\lambda}^{\ast}}\rightharpoonup v=
\int_{\mathbb{H}^N}\frac{|u(\eta)|^{Q_{\lambda}^{\ast}}}{|\eta^{-1}\xi|^{\lambda}}d\eta|u(\xi)|^{Q_{\lambda}^{\ast}}+\sum_{j\in J}v_{j}\delta_{x_{j}}
\end{equation}
in the sense of measure, where $x_{j}\in\mathbb{H}^N$ and $\delta_{x_{j}}$ is the Dirac mass at $x_{j}$.
In addition, we have
\begin{equation}\label{4.3}
\kappa_j\geq H_{Q_{\lambda}^{\ast}}v_j^{\frac{p}{Q_{\lambda}^{\ast}}},\
\hbox{for every}\
 j\in J.
\end{equation}

Next, we shall prove that $v_{j}=0\ \mbox{for every}\ j\in J$.
For this purpose, let $\xi_j$ be a singular point of the measures $\kappa$ and $v$,
and we define $\psi_{\varepsilon,j}=\psi(\frac{\xi-\xi_j}{\varepsilon})$ as a cut-off function.
Moreover, the hypotheses $0\leq\psi(\xi)\leq1$,
\begin{equation*}
\begin{cases}
\psi(\xi)=1 &\mbox{in}   \   B_{1}(0),\\
\psi(\xi)=0 &\mbox{in}   \   \mathbb{H}^N\backslash B_{2}(0),\\
|\nabla_{H}\psi(\xi)|\leq2 &\mbox{in}   \   \mathbb{H}^N
\end{cases}
\end{equation*}
hold, where $\psi\in C_{0}^{\infty}(\mathbb{H}^N)$.
Now by the boundedness of $\{\psi_{\varepsilon,j}u_{n}\}$ in $W_\mu$,
we have $\langle I_{\mu}'(u_{n}),\psi_{\varepsilon,j}u_{n}\rangle\rightarrow0$
as $n\rightarrow\infty$.
Furthermore, we have
\begin{align}\label{4.4}
\begin{split}
M(\|u_n\|_{\mu}^{p})\langle L(u_n),\psi_{\varepsilon,j}u_{n}\rangle
=\int_{\mathbb{H}^{N}}\int_{\mathbb{H}^N}\frac{|u_n(\xi)|^{Q_{\lambda}^{\ast}}|u_n(\eta)|^{Q_{\lambda}^{\ast}}\psi_{\varepsilon,j}}{|\eta^{-1}\xi|^{\lambda}}d\eta d\xi
+\int_{\mathbb{H}^N}f(\xi,u_n)u_n\psi_{\varepsilon,j}d\xi+o(1),
\end{split}
\end{align}
where
\begin{align*}
\begin{split}
\langle L(u_n),\psi_{\varepsilon,j}u_{n}\rangle
=&\mu\int\int_{\mathbb{H}^{2N}}\frac{|u_n(\xi)-u_n(\eta)|^{p}\psi_{\varepsilon,j}(\xi)}{|\eta^{-1}\xi|^{N+ps}}d\xi d\eta+\int_{\mathbb{H}^{N}}V(\xi)|u_n|^{p}\psi_{\varepsilon,j}d\xi\\
&+\mu\int\int_{\mathbb{H}^{2N}}\frac{|u_n(\xi)-u_n(\eta)|^{p-2}(u_n(\xi)-u_n(\eta))u_n(\eta)(\psi_{\varepsilon,j}(\xi)-\psi_{\varepsilon,j}(\eta))}{|\eta^{-1}\xi|^{N+ps}}d\xi
d\eta.
\end{split}
\end{align*}
Similarly to the proof of Xiang et al. \cite[Lemma 2.3]{x1}, we have
$$\lim_{\varepsilon\rightarrow0}\limsup_{n\rightarrow\infty}\left(\int\int_{\mathbb{H}^{2N}}\frac{|\left(\psi_{\varepsilon,j}(\xi)-\psi_{\varepsilon,j}(\eta)\right)
u_n(\xi)|^{p}}{|\eta^{-1}\xi|^{N+ps}}d\xi
d\eta\right)^{\frac{1}{p}}=0.$$ It follows from the H\"{o}lder
inequality that
\begin{align}\label{4.5}
\begin{split}
&\lim_{\varepsilon\rightarrow0}\limsup_{n\rightarrow\infty}M(\|u_n\|_{\mu}^{p})\left(\mu\int\int_{\mathbb{H}^{2N}}\frac{|u_n(\xi)-u_n(\eta)|^{p-2}
(u_n(\xi)-u_n(\eta))u_n(\eta)(\psi_{\varepsilon,j}(\xi)-\psi_{\varepsilon,j}(\eta))}{|\eta^{-1}\xi|^{N+ps}}d\xi d\eta\right)\\
\ &\leq
C\lim_{\varepsilon\rightarrow0}\limsup_{n\rightarrow\infty}\left(\int\int_{\mathbb{H}^{2N}}\frac{|u_n(\xi)-u_n(\eta)|^{p}
}{|\eta^{-1}\xi|^{N+ps}}d\xi
d\eta\right)^{\frac{p-1}{p}}\left(\int\int_{\mathbb{H}^{2N}}\frac{|u_n(\eta)(\psi_{\varepsilon,j}(\xi)-\psi_{\varepsilon,j}(\eta))|^{p}
}{|\eta^{-1}\xi|^{N+ps}}d\xi d\eta\right)^{\frac{1}{p}}\\
&\ \leq
C\lim_{\varepsilon\rightarrow0}\limsup_{n\rightarrow\infty}\left(\int\int_{\mathbb{H}^{2N}}\frac{|u_n(\eta)(\psi_{\varepsilon,j}(\xi)-\psi_{\varepsilon,j}(\eta))|^{p}
}{|\eta^{-1}\xi|^{N+ps}}d\xi d\eta\right)^{\frac{1}{p}}=0.
\end{split}
\end{align}
Thus, by \eqref{4.1}, \eqref{4.3}, and condition $(M)$, we get
\begin{align}\label{4.6}
\begin{split}
&\lim_{\varepsilon\rightarrow0}\limsup_{n\rightarrow\infty}M(\|u_n\|_{\mu}^{p})\left(\mu\int\int_{\mathbb{H}^{2N}}\frac{|u_n(\xi)-u_n(\eta)|^{p}
\psi_{\varepsilon,j}(\xi)}{|\eta^{-1}\xi|^{N+ps}}d\xi d\eta+\int_{\mathbb{H}^{N}}V(\xi)|u_n|^{p}\psi_{\varepsilon,j}d\xi\right)\\
&\geq\lim_{\varepsilon\rightarrow0}\limsup_{n\rightarrow\infty}m_0\left(\mu\int\int_{\mathbb{H}^{2N}}\frac{|u_n(\xi)-u_n(\eta)|^{p}
\psi_{\varepsilon,j}(\xi)}{|\eta^{-1}\xi|^{N+ps}}d\xi d\eta\right)^{\tau}\\
&\geq\lim_{\varepsilon\rightarrow0}m_0\left(\mu\int\int_{\mathbb{H}^{2N}}\frac{|u(\xi)-u(\eta)|^{p}
\psi_{\varepsilon,j}(\xi)}{|\eta^{-1}\xi|^{N+ps}}d\xi d\eta+\mu\kappa_j\right)^{\tau}
=m_0(\mu\kappa_j)^{\tau}\geq m_0(\mu
H_{Q_{\lambda}^{\ast}}v_j^{\frac{p}{Q_{\lambda}^{\ast}}})^{\tau}.
\end{split}
\end{align}
Moreover, it follows from \eqref{4.2} that
\begin{equation}\label{4.7}
\lim_{\varepsilon\rightarrow0}\limsup_{n\rightarrow\infty}\int_{\mathbb{H}^N}\int_{\mathbb{H}^N}
\frac{|u_{n}(\eta)|^{Q_{\lambda}^{\ast}}|u_{n}(\xi)|^{Q_{\lambda}^{\ast}}\psi_{\varepsilon,j}}{|\eta^{-1}\xi|^{\lambda}}d\eta
d\xi
=\lim_{\varepsilon\rightarrow0}\int_{\mathbb{H}^N}\int_{\mathbb{H}^N}\frac{|u(\eta)|^{Q_{\lambda}^{\ast}}|u(\xi)|^{Q_{\lambda}^{\ast}}
\psi_{\varepsilon,j}}{|\eta^{-1}\xi|^{\lambda}}d\eta d\xi+v_j=v_j
\end{equation}
and by Lemma \ref{lem2.1}, since $W_\mu\hookrightarrow L^{\gamma}(\mathbb{H}^N)$ is a compact embedding for every $\gamma\in[1,Q_{\lambda}^{\ast})$, we have
\begin{equation}\label{4.8}
\lim_{\varepsilon\rightarrow0}\limsup_{n\rightarrow\infty}\int_{\mathbb{H}^N}f(\xi,u_n)u_n\psi_{\varepsilon,j}d\xi
=\lim_{\varepsilon\rightarrow0}\int_{\mathbb{H}^N}f(\xi,u)u\psi_{\varepsilon,j}d\xi=0.
\end{equation}
Now we put \eqref{4.5}-\eqref{4.8} into
\eqref{4.4} and obtain
$$v_j\geq m_0(\mu H_{Q_{\lambda}^{\ast}}v_j^{\frac{p}{Q_{\lambda}^{\ast}}})^{\tau}=m_0\mu^{\tau} H_{Q_{\lambda}^{\ast}}^{\tau}v_j^{\frac{p\tau}{Q_{\lambda}^{\ast}}},$$
which implies that either $v_j=0$ or $v_j\geq (m_0\mu^{\tau} H_{Q_{\lambda}^{\ast}}^{\tau})^{\frac{Q_{\lambda}^{\ast}}{Q_{\lambda}^{\ast}-p\tau}}$.

Now we can prove the possible concentration of mass at infinity.
Similarly to the above argument, we define $\phi_R\in
C_{0}^{\infty}(\mathbb{H}^N)$ as a cut-off function. Moreover, the
hypotheses $0\leq\phi_R\leq1$,
\begin{equation*}
\begin{cases}
\phi_R(\xi)=0 &\mbox{in}   \   B_{R}(0),\\
\phi_R(\xi)=1 &\mbox{in}   \   \mathbb{H}^N\backslash B_{2R}(0),\\
|\nabla_{H}\phi_R(\xi)|\leq\frac{2}{R} &\mbox{in}   \   \mathbb{H}^N
\end{cases}
\end{equation*}
hold. Next, again
by Lemma \ref{lem2.4}, one has
\begin{equation}\label{4.9}
\kappa_\infty=\lim_{R\rightarrow\infty}\limsup_{n\rightarrow\infty}\int\int_{\mathbb{H}^{2N}}
\frac{|u_n(\xi)-u_n(\eta)|^{p}\phi_R(\xi)^{p}}{|\xi-\eta|^{N+ps}}d\eta
d\xi
\end{equation}
and
\begin{align}\label{4.10}
\begin{split}
v_\infty&=\lim_{R\rightarrow\infty}\limsup_{n\rightarrow\infty}\int
\int_{\mathbb{H}^{2N}}\frac{|u_{n}(\eta)|^{Q_{\lambda}^{\ast}}|u_{n}(\xi)|^{Q_{\lambda}^{\ast}}\phi_R(\eta)^{2Q_{\lambda}^{\ast}}}{|\eta^{-1}\xi|^{\lambda}}d\eta d\xi\\
&=\lim_{R\rightarrow\infty}\limsup_{n\rightarrow\infty}\int
\int_{\mathbb{H}^{2N}}\frac{|u_{n}(\eta)|^{Q_{\lambda}^{\ast}}|u_{n}(\xi)|^{Q_{\lambda}^{\ast}}\phi_R(\eta)}{|\eta^{-1}\xi|^{\lambda}}d\eta d\xi.
\end{split}
\end{align}
In addition, one has
\begin{equation}\label{4.11}
\kappa_\infty\geq H_{Q_{\lambda}^{\ast}}v_\infty^{\frac{p}{Q_{\lambda}^{\ast}}}.
\end{equation}
Since $\langle I_\mu'(u_n),\phi_Ru_n\rangle\rightarrow0$ as
$n\rightarrow\infty$, we get
\begin{align}\label{4.12}
\begin{split}
M(\|u_n\|_{\mu}^{p})\langle L(u_n),\phi_Ru_n\rangle
=\int_{\mathbb{H}^{N}}\int_{\mathbb{H}^N}\frac{|u_n(\xi)|^{Q_{\lambda}^{\ast}}|u_n(\eta)|^{Q_{\lambda}^{\ast}}\phi_Ru_n}{|\eta^{-1}\xi|^{\lambda}}d\eta d\xi
+\int_{\mathbb{H}^N}f(\xi,u_n)u_n\phi_Rd\xi+o(1),
\end{split}
\end{align}
where
\begin{align*}
\begin{split}
\langle L(u_n),\phi_Ru_n\rangle
=\mu\int\int_{\mathbb{H}^{2N}}\frac{|u_n(\xi)-u_n(\eta)|^{p}\phi_R(\xi)}{|\eta^{-1}\xi|^{N+ps}}d\xi d\eta+\int_{\mathbb{H}^{N}}V(\xi)|u_n|^{p}\phi_Rd\xi\\
+\mu\int\int_{\mathbb{H}^{2N}}\frac{|u_n(\xi)-u_n(\eta)|^{p-2}(u_n(\xi)-u_n(\eta))u_n(\eta)(\phi_R(\xi)-\phi_R(\eta))}{|\eta^{-1}\xi|^{N+ps}}d\xi
d\eta.
\end{split}
\end{align*}
Similarly to the proof of Xiang et al. \cite{x1}, one has
$$\lim_{R\rightarrow\infty}\limsup_{n\rightarrow\infty}\left(\int\int_{\mathbb{H}^{2N}}\frac{|\left(\phi_R(\xi)-\phi_R(\eta)\right)
u_n(\xi)|^{p}}{|\xi-\eta|^{N+ps}}d\xi
d\eta\right)^{\frac{1}{p}}=0.$$ It follows from the H\"{o}lder
inequality that
\begin{align}\label{4.13}
\begin{split}
&\lim_{R\rightarrow\infty}\limsup_{n\rightarrow\infty}M(\|u_n\|_{\mu}^{p})\left(\mu\int\int_{\mathbb{H}^{2N}}\frac{|u_n(\xi)-u_n(\eta)|^{p-2}
(u_n(\xi)-u_n(\eta))u_n(\eta)(\phi_R(\xi)-\phi_R(\eta))}{|\eta^{-1}\xi|^{N+ps}}d\xi d\eta\right)\\
\ &\leq
C\lim_{R\rightarrow\infty}\limsup_{n\rightarrow\infty}\left(\int\int_{\mathbb{H}^{2N}}\frac{|u_n(\xi)-u_n(\eta)|^{p}
}{|\eta^{-1}\xi|^{N+ps}}d\xi
d\eta\right)^{\frac{p-1}{p}}\left(\int\int_{\mathbb{H}^{2N}}\frac{|u_n(\eta)(\phi_R(\xi)-\phi_R(\eta))|^{p}
}{|\eta^{-1}\xi|^{N+ps}}d\xi d\eta\right)^{\frac{1}{p}}\\
&\ \leq
C\lim_{R\rightarrow\infty}\limsup_{n\rightarrow\infty}\left(\int\int_{\mathbb{H}^{2N}}\frac{|u_n(\eta)(\phi_R(\xi)-\phi_R(\eta))|^{p}
}{|\eta^{-1}\xi|^{N+ps}}d\xi d\eta\right)^{\frac{1}{p}}=0.
\end{split}
\end{align}
For the first term on the right hand side of \eqref{4.12}, we get by
\eqref{2.9}, \eqref{4.9},
 and condition $(M)$
\begin{align}\label{4.14}
\begin{split}
&\lim_{R\rightarrow\infty}\limsup_{n\rightarrow\infty}M(\|u_n\|_{\mu}^{p})\left(\mu\int\int_{\mathbb{H}^{2N}}\frac{|u_n(\xi)-u_n(\eta)|^{p}\phi_R(\xi)}{|\eta^{-1}\xi|^{N+ps}}d\xi d\eta+\int_{\mathbb{H}^{N}}V(\xi)|u_n|^{p}\phi_Rd\xi\right)\\
&\geq\lim_{R\rightarrow\infty}\limsup_{n\rightarrow\infty}m_0\left(\mu\int\int_{\mathbb{H}^{2N}}\frac{|u_n(\xi)-u_n(\eta)|^{p}\phi_R(\xi)^{p}}
{|\eta^{-1}\xi|^{N+ps}}d\xi d\eta\right)^{\tau-1}
\left(\mu\int_{\{\xi\in\mathbb{H^{N}}:|\xi|>2R\}}\int_{\mathbb{H}^{N}}\frac{|u_n(\xi)-u_n(\eta)|^{p}}{|\eta^{-1}\xi|^{N+ps}}d\xi d\eta\right)\\
&\geq
m_0\mu^{\tau}\kappa_\infty^{\tau-1}(\int_{\mathbb{H}^{N}}d\kappa+\kappa_\infty)\geq
m_0\mu^{\tau}\kappa_\infty^{\tau}.
\end{split}
\end{align}
For the second term on the right hand side of \eqref{4.12},
since $(f_1)'$ and $S_\mu\hookrightarrow L^{\gamma}(\mathbb{H}^N)$ is a compact embedding for every $\gamma\in[1,Q_{\lambda}^{\ast})$,
it is easy to get
\begin{equation}\label{4.15}
\lim_{R\rightarrow\infty}\limsup_{n\rightarrow\infty}\int_{\mathbb{H}^N}f(\xi,u_n)u_n\phi_Rdx
=\lim_{R\rightarrow\infty}\int_{\mathbb{H}^N}f(\xi,u)u\phi_Rd\xi=0.
\end{equation}
Now we put \eqref{4.10} 
and
 \eqref{4.13}-\eqref{4.15} into \eqref{4.12} to get
$$v_\infty\geq m_0(\mu H_{Q_{\lambda}^{\ast}}v_\infty^{\frac{p}{Q_{\lambda}^{\ast}}})^{\tau}=m_0\mu^{\tau} H_{Q_{\lambda}^{\ast}}^{\tau}v_\infty^{\frac{p\tau}{Q_{\lambda}^{\ast}}},$$
which implies that either $v_\infty=0$ or $v_\infty\geq (m_0\mu^{\tau} H_{Q_{\lambda}^{\ast}}^{\tau})^{\frac{Q_{\lambda}^{\ast}}{Q_{\lambda}^{\ast}-p\tau}}$.

In the sequel, we shall prove that $v_j\geq (m_0\mu^{\tau}
H_{Q_{\lambda}^{\ast}}^{\tau})^{\frac{Q_{\lambda}^{\ast}}{Q_{\lambda}^{\ast}-p\tau}}$
and $v_\infty\geq (m_0\mu^{\tau}
H_{Q_{\lambda}^{\ast}}^{\tau})^{\frac{Q_{\lambda}^{\ast}}{Q_{\lambda}^{\ast}-p\tau}}$
is impossible. Indeed, if $v_j\geq (m_0\mu^{\tau}
H_{Q_{\lambda}^{\ast}}^{\tau})^{\frac{Q_{\lambda}^{\ast}}{Q_{\lambda}^{\ast}-p\tau}}$,
by Lemma \ref{lem2.4} and \eqref{4.3}, we would have
\begin{align*}
\begin{split}
c&=\lim_{n\rightarrow\infty}I_{\mu}(u_n)-\frac{1}{q_0}\langle I_{\mu}'(u_n),u_n\rangle\\
&\geq\left(\frac{m_0}{p\tau}-\frac{m_1}{q_0}\right)\left(\mu\int\int_{\mathbb{H}^{2N}}
\frac{|u_n(\xi)-u_n(\eta)|^{p}}{|\eta^{-1}\xi|^{N+ps}}d\xi
d\eta\right)^{\tau}
\\
&\quad+\left(\frac{1}{q_0}-\frac{1}{2Q_{\lambda}^{\ast}}\right)\int_{\mathbb{H}^{N}}\int_{\mathbb{H}^N}
\frac{|u_n(\eta)|^{Q_{\lambda}^{\ast}}|u_n(\xi)|^{Q_{\lambda}^{\ast}}}{|\eta^{-1}\xi|^{\lambda}}d\eta d\xi\\
&\geq\left(\frac{m_0}{p\tau}-\frac{m_1}{q_0}\right)(\mu\kappa_j)^{\tau}+\left(\frac{1}{q_0}-\frac{1}{2Q_{\lambda}^{\ast}}\right)v_j\\
&\geq\left(\frac{1}{\tau p}-\frac{1}{2Q_{\lambda}^{\ast}}+\frac{1}{q_0}(1-\frac{m_1}{m_0})\right)(m_0\mu^{\tau} H_{Q_{\lambda}^{\ast}}^{\tau})^{\frac{Q_{\lambda}^{\ast}}{Q_{\lambda}^{\ast}-p\tau}}
=\rho(m_0\mu^{\tau} H_{Q_{\lambda}^{\ast}}^{\tau})^{\frac{Q_{\lambda}^{\ast}}{Q_{\lambda}^{\ast}-p\tau}},
\end{split}
\end{align*}
where $\rho=\frac{1}{\tau p}-\frac{1}{2Q_{\lambda}^{\ast}}+\frac{1}{q_0}(1-\frac{m_1}{m_0})$,
which contradicts
$c\in(0, \rho(m_0\mu^{\tau}H_{Q_{\lambda}^{\ast}}^{\tau})^{\frac{Q_{\lambda}^{\ast}}{Q_{\lambda}^{\ast}-\tau p}})$.
This means that $v_j=0$ for every $j\in J$. Similarly, $v_\infty=0$.
Thus, it follows \eqref{2.10} that
$$
\lim_{n\rightarrow\infty}\int_{\mathbb{H}^N}\int_{\mathbb{H}^N}\frac{|u_n(\eta)|^{Q_{\lambda}^{\ast}}|u_n(\xi)|^{Q_{\lambda}^{\ast}}}{|\eta^{-1}\xi|^{\lambda}}d\eta
d\xi=
\int_{\mathbb{H}^N}\int_{\mathbb{H}^N}\frac{|u(\eta)|^{Q_{\lambda}^{\ast}}|u(\xi)|^{Q_{\lambda}^{\ast}}}{|\eta^{-1}\xi|^{\lambda}}d\eta
d\xi.
$$
Invoking
 the Br\'{e}zis-Lieb Lemma, we obtain
\begin{equation}\label{4.16}
\lim_{n\rightarrow\infty}\int_{\mathbb{H}^N}\int_{\mathbb{H}^N}\frac{|u_{n}(\eta)-u(\eta)|^{Q_{\lambda}^{\ast}}
|u_{n}(\xi)-u(\xi)|^{Q_{\lambda}^{\ast}}}{|\eta^{-1}\xi|^{\lambda}}d\eta d\xi=0.
\end{equation}

Finally, we prove that $u_n\rightarrow u$ in $S_\mu$. Take
$\{u_n\}_n\subset S_\mu$ to be
a
 $(PS)_c$ sequence of the functional
$I_{\mu}$, and define
\begin{align*}
\begin{split}
\langle L(u),\varphi\rangle
=\mu\int\int_{\mathbb{H}^{2N}}\frac{|u(\xi)-u(\eta)|^{p-2}(u(\xi)-u(\eta))(\varphi(\xi)-\varphi(\eta))}{|\eta^{-1}\xi|^{N+ps}}d\xi
d\eta+\int_{\mathbb{H}^{N}}V(\xi)|u|^{p-2}u\varphi d\xi
\end{split}
\end{align*}
for every $\varphi\in S_\mu$.
Then, we have
\begin{align}\label{4.17}
\begin{split}
o(1)=&\langle I_{\mu}'(u_n)-I_{\mu}'(u),u_n-u\rangle\\
=&M(\|u_n\|_{\mu}^{p})\langle L(u_n),u_n-u\rangle-M(\|u\|_{\mu}^{p})\langle L(u),u_n-u\rangle\\
&\
-\int_{\mathbb{H}^{N}}\int_{\mathbb{H}^N}\frac{(|u_n(\xi)|^{Q_{\lambda}^{\ast}}|u_n(\eta)|^{Q_{\lambda}^{\ast}-2}u_n(\eta)-|u(\xi)|^{Q_{\lambda}^{\ast}}
|u(\eta)|^{Q_{\lambda}^{\ast}-2}u(\eta))(u_n(\eta)-u(\eta))}{|\eta^{-1}\xi|^{\lambda}}d\eta d\xi\\
&\ -\int_{\mathbb{H}^{N}}(f(\xi,u_n)-f(\xi,u))(u_n-u)d\xi.\\
\end{split}
\end{align}
For the fourth term on the right
hand side
 of \eqref{4.17}, similarly to \eqref{3.3}, we have
\begin{equation}\label{4.18}
\lim_{n\rightarrow\infty}\int_{\mathbb{H}^{N}}(f(\xi,u_n)-f(\xi,u))(u_n-u)d\xi=0.
\end{equation}
For the third term on the right
hand side
 of \eqref{4.17}, it follows by \eqref{3.1} and \eqref{4.16} that
\begin{align}\label{4.19}
\begin{split}
\lim_{n\rightarrow\infty}\int_{\mathbb{H}^{N}}\int_{\mathbb{H}^N}\frac{(|u_n(\xi)|^{Q_{\lambda}^{\ast}}|u_n(\eta)|^{Q_{\lambda}^{\ast}-2}u_n(\eta)-|u(\xi)|^{Q_{\lambda}^{\ast}}
|u(\eta)|^{Q_{\lambda}^{\ast}-2}u(\eta))(u_n(\eta)-u(\eta))}{|\eta^{-1}\xi|^{\lambda}}d\eta d\xi=0.
\end{split}
\end{align}
Now we put \eqref{4.18}
and
 \eqref{4.19} 
 into
  \eqref{4.17} and obtain
\begin{align}\label{4.20}
\begin{split}
o(1)=&M(\|u_n\|_{\mu}^{p})(\langle L(u_n),u_n-u\rangle-\langle L(u),u_n-u\rangle)+M(\|u_n\|_{\mu}^{p})\langle L(u),u_n-u\rangle\\
&\ -M(\|u_n\|_{\mu}^{p})\langle L(u),u_n-u\rangle.\\
\end{split}
\end{align}
Since $u_n\rightharpoonup u$ and $\{u_n\}_n$ is bounded in $S_\mu$,
we can deduce that
\begin{equation*}
\lim_{n\rightarrow\infty}M(\|u_n\|_{\mu}^{p})\langle L(u),u_n-u\rangle=0,
\end{equation*}
hence
\begin{equation*}
\lim_{n\rightarrow\infty}M(\|u_n\|_{\mu}^{p})(\langle L(u_n),u_n-u\rangle-\langle L(u),u_n-u\rangle)=0,
\end{equation*}
so since $\inf_{n\in N}\|u_n\|_{\mu}=l>0,$
we get
\begin{align}\label{4.21}
\begin{split}
o(1)&=(\langle L(u_n),u_n-u\rangle-\langle L(u),u_n-u\rangle)\\
&=\mu\langle u_n-u,u_n-u\rangle_{s,p}^{p}+\int_{\mathbb{H}^{N}}V(\xi)(|u_n|^{p-2}u_n-|u|^{p-2}u)(u_n-u)d\xi\\
&=B_1+B_2,
\end{split}
\end{align}
where
$$B_1=\mu\langle u_n-u,u_n-u\rangle_{s,p}^{p}, \quad
B_2=\int_{\mathbb{H}^{N}}V(\xi)(|u_n|^{p-2}u_n-|u|^{p-2}u)(u_n-u)d\xi,$$
and
\begin{equation*}
\langle
u,\varphi\rangle_{s,p}^{p}=\int\int_{\mathbb{H}^{2N}}\frac{|u(\xi)-u(\eta)|^{p-2}(u(\xi)-u(\eta))
(\varphi(\xi)-\varphi(\eta))}{|\eta^{-1}\xi|^{N+ps}}d\xi d\eta.
\end{equation*}
So this gives us $B_1\geq0$. We invoke some elementary inequalities
(see, e.g., Kichenassamy and Veron \cite{sk}): For every $p>1$ there
exist  positive constants $C_1=C(p,n)>0$ and $C_2=C(p,n)>0$ such
that
\begin{equation}\label{4.22}
|\xi-\eta|^{p}\leq
\begin{cases}
C_1(|\xi|^{p-2}\xi-|\eta|^{p-2}\eta)(\xi-\eta) &\mbox{if}   \   p\geq2,\\
C_2[(|\xi|^{p-2}\xi-|\eta|^{p-2}\eta)(\xi-\eta)]^{\frac{p}{2}}(|\xi|^{p}+|\eta|^{p})^{\frac{2-p}{2}} &\mbox{if}   \   1<p<2\\
\end{cases}
\end{equation}
for every $\xi,\eta\in\mathbb{R}$. Thus, $B_2\geq0$. It follows by
 \eqref{4.21} that $B_1=B_2=o(1)$.

For $p\geq2$, one has
\begin{align*}
\begin{split}
\mu[u_n-u]_{s,p}^{p}&=\mu\int\int_{\mathbb{H}^{2N}}\frac{|(u_n(\xi)-u_n(\eta))-(u(\xi)-u(\eta))|^{p}}{|\eta^{-1}\xi|^{N+ps}}d\xi d\eta\\
&\leq \tcr{C_1}\mu\left(\langle u_n,u_n-u\rangle_{s,p}^{p}-\langle
u,u_n-u\rangle_{s,p}^{p}\right)=o(1)
\end{split}
\end{align*}
and
\begin{align*}
\begin{split}
\|u_n-u\|_V^{p}\leq \tcr{C_1}\int_{\mathbb{H}^{N}}V(\xi)(|u_n|^{p-2}u_n-|u|^{p-2}u)(u_n-u)d\xi=o(1),
\end{split}
\end{align*}
which
implies
 that $\|u_n-u\|_\mu=o(1)$.

For $1<p<2$, we utilize the following inequality:
$$(a+b)^{s}\leq a^{s}+b^{s},\quad \mbox{for every}\ a,b>0,\ s\in(0,1),$$
and by \eqref{4.22}, $B_1=o(1)$ and the H\"{o}lder inequality,
one has
\begin{align*}
\begin{split}
\mu[u_n-u]_{s,p}^{p}&=\mu\int\int_{\mathbb{H}^{2N}}\frac{|(u_n(\xi)-u_n(\eta))-(u(\xi)-u(\eta))|^{p}}{|\eta^{-1}\xi|^{N+ps}}d\xi d\eta\\
&\leq \tcr{C_2}\mu\left(\langle u_n,u_n-u\rangle_{s,p}^{p}-\langle
u,u_n-u\rangle_{s,p}^{p}\right)^{\frac{p}{2}}\left([u_n]_{s,p}^{\frac{p(2-p)}{2}}+
[u]_{s,p}^{\frac{p(2-p)}{2}}\right)\\
&\leq \tcr{C_2}\mu\left(\langle u_n,u_n-u\rangle_{s,p}^{p}-\langle u,u_n-u\rangle_{s,p}^{p}\right)^{\frac{p}{2}}=o(1).
\end{split}
\end{align*}
Similarly
 to \eqref{4.22} and $B_2=o(1)$, one has
\begin{align*}
\begin{split}
\|u_n-u\|_V^{p}&\leq C\int_{\mathbb{H}^{N}}V(\xi)[(|u_n|^{p-2}u_n-|u|^{p-2}u)(u_n-u)]^{\frac{p}{2}}(|u_n|^{p}+|u|^{p})^{\frac{2-p}{2}}d\xi\\
&\leq \tcr{C_2}\left(\int_{\mathbb{H}^{N}}V(\xi)(|u_n|^{p-2}u_n-|u|^{p-2}u)(u_n-u)\right)^{\frac{p}{2}}
\left(\int_{\mathbb{H}^{N}}V(\xi)(|u_n|^{p}+|u|^{p})d\xi\right)^{\frac{2-p}{2}}\\
&\leq \tcr{C_2}\left(\int_{\mathbb{H}^{N}}V(\xi)(|u_n|^{p-2}u_n-|u|^{p-2}u)(u_n-u)\right)^{\frac{p}{2}}=o(1),
\end{split}
\end{align*}
which
 implies $\|u_n-u\|_{\mu}=o(1)$.

Case II: $\inf_{n\in N}\|u_n\|_{\mu}=l=0$.
If $0$ is an accumulation point of the sequence $\{u_n\}_n$, then there is a subsequence
of $\{u_n\}_n$ that converges strongly to $u=0$ so we get the desired result.
If $0$ is an isolated point of
the sequence $\{u_n\}_n$, then there is a subsequence, still denoted by $\{u_n\}_n$ satisfying $\inf_{n\in N}\|u_n\|_{\mu}=l>0$, which
was considered in
 Case I. This completes the proof of Lemma \ref{lem4.2}.
\end{proof}

Under assumptions $M(\cdot)$, $V(\cdot),$ and $f(\cdot,\cdot)$, we
can now prove that the function has the mountain pass geometry.

\begin{lemma}\label{lem4.3}
Let $\tau\in(1,\frac{Q_\lambda^\ast}{p})$ and suppose that the
Kirchhoff function $M(\cdot)$ satisfies condition $(M)$. Assume that
$f(\cdot,\cdot)$ satisfies condition $(f_{1})'$, and the potential
function $V(\xi)$ satisfies conditions $(V_1)$ and $(V_2)$. Then for
every $\mu>0$, there exist
 $\alpha, \sigma>0$ satisfying $I_{\mu}(u)>0$ for
$u\in B_\sigma\setminus\{0\}$, and $I_{\mu}(u)\geq\alpha$ for every $u\in S_\mu$
with $\|u\|_{\mu}=\sigma$, where $B_\sigma=\{u\in S_\mu:\|u\|_{\mu}<\sigma\}$.
\end{lemma}
\begin{proof}
By condition $(f_1)'$, there is $C_{\varepsilon}>0\ (\mbox{for
every}\ \varepsilon>0)$ satisfying
$$F(\xi,t)\leq \varepsilon|t|^{\tau p}+C_{\varepsilon}|t|^{q}\quad \mbox{for a.e.}  \ \xi\in \mathbb{H}^{N}\mbox{and all }\ t\in\mathbb{R}.$$
For any $u\in W_{\mu}$, by condition $(M)$, we have
\begin{equation*}
I_{\mu}(u)\geq\frac{m_0}{p\tau}\|u\|_{\mu}^{p\tau}-\frac{1}{2Q_{\lambda}^{\ast}}\int_{\mathbb{H}^{N}}\int_{\mathbb{H}^N}
\frac{|u(\eta)|^{Q_{\lambda}^{\ast}}|u(\xi)|^{Q_{\lambda}^{\ast}}}{|\eta^{-1}\xi|^{\lambda}}d\eta d\xi-\varepsilon\int_{\mathbb{H}^{N}}|u|^{\tau p}d\xi-C_\varepsilon\int_{\mathbb{H}^{N}}|u|^{q}d\xi.
\end{equation*}

By taking $\varepsilon\in (0,\frac{m_0}{2\tau pc})$ and applying the definition of $H_{Q_{\lambda}^{\ast}}$,
we have
\begin{align*}
\begin{split}
I_{\mu}(u)&\geq\frac{m_0}{p\tau}\|u\|_{\mu}^{p\tau}-\frac{\mu^{-1}}{2Q_{\lambda}^{\ast}}H_{Q_{\lambda}^{\ast}}^{-Q_{\lambda}^{\ast}/p}\|u\|_{\mu}^
{Q_{\lambda}^{\ast}}-C\varepsilon\|u\|_{\mu}^{p\tau}-CC_\varepsilon\|u\|_{\mu}^{q}\\
&\geq\left(\frac{m_0}{2\tau p}-\frac{\mu^{-1}}{2Q_{\lambda}^{\ast}}H_{Q_{\lambda}^{\ast}}^{-Q_{\lambda}^{\ast}/p}\|u\|_{\mu}^{Q_{\lambda}^{\ast}-\tau p}-CC_\varepsilon\|u\|_{\mu}^{q-\tau p}\right)\|u\|_{\mu}^{\tau p},\\
\end{split}
\end{align*}
in the last inequality, we used the fractional Sobolev embedding $|u|_{\tau p}\leq C\|u\|_{\mu}$ and $|u|_q\leq C\|u\|_{\mu}$.

Next, we define
$$g(t)=\frac{m_0}{2\tau p}-\frac{\mu^{-1}}{2Q_{\lambda}^{\ast}}H_{Q_{\lambda}^{\ast}}^{-Q_{\lambda}^{\ast}/p}t^{Q_{\lambda}^{\ast}-\tau p}-CC_\varepsilon t^{q-\tau p}\quad \mbox{for every}\ t\geq0.$$
Since $Q_{\lambda}^{\ast}>\tau p$ and $q>\tau p$, it is clear that $\lim_{t\rightarrow0^{+}}g(t)=\frac{m_0}{2\tau p}$.
Pick $\sigma=\|u\|_{\mu}$ small enough such that
$$\frac{\mu^{-1}}{2Q_{\lambda}^{\ast}}H_{Q_{\lambda}^{\ast}}^{-Q_{\lambda}^{\ast}/p}\sigma^{Q_{\lambda}^{\ast}-\tau p}+CC_\varepsilon \sigma^{q-\tau p}<\frac{m_0}{2\tau p},$$
to get
$$I_{\mu}(u)\geq g(\sigma)\sigma^{\tau p}=\alpha.$$
The proof of Lemma \ref{lem4.3} is thus complete.
\end{proof}

\begin{lemma}\label{lem4.4}
Let $\tau\in(1,\frac{Q_\lambda^\ast}{p})$ and suppose that Kirchhoff
function $M(\cdot)$ satisfies condition $(M)$. Assume that
$f(\cdot,\cdot)$ satisfies condition $(f_{1})'$, and the potential
function $V(\xi)$ satisfies conditions $(V_1)$ and $(V_2)$. Then for
every $\mu>0$, there exists
 $e\in S_\mu$ with $\|e\|_{\mu}>\sigma$ satisfying $I_{\mu}(e)<0$,
where $\sigma$ is given by Lemma \ref{lem4.3}.
\end{lemma}

\begin{proof}
By condition $(f_{2})'$, we know that $F(\xi,t)\geq0$ for a.e.
$\xi\in \mathbb{H}^{N}$. Choose a function $u_0\in S_\mu$ satisfying
$$\|u_0\|_{\mu}=1\quad\mbox{and}\quad\frac{1}{2Q_{\lambda}^{\ast}}\int_{\mathbb{H}^{N}}\int_{\mathbb{H}^N}
\frac{|u_0(\eta)|^{Q_{\lambda}^{\ast}}|u_0(\xi)|^{Q_{\lambda}^{\ast}}}{|\eta^{-1}\xi|^{\lambda}}d\eta d\xi>0.$$
By condition $(M)$ and $F(\xi,t)\geq0$ for a.e. $\xi\in \mathbb{H}^{N}$, one has
\begin{align*}
\begin{split}
I_{\mu}(tu_0)&\leq\frac{m_1t^{\tau p}}{p\tau}\|u_0\|_{\mu}^{p\tau}-\frac{t^{2Q_{\lambda}^{\ast}}}{2Q_{\lambda}^{\ast}}\int_{\mathbb{H}^{N}}\int_{\mathbb{H}^N}
\frac{|u_0(\eta)|^{Q_{\lambda}^{\ast}}|u_0(\xi)|^{Q_{\lambda}^{\ast}}}{|\eta^{-1}\xi|^{\lambda}}d\eta d\xi.\\
\end{split}
\end{align*}
Since $2Q_{\lambda}^{\ast}>\tau p$,
 there exists
  $t\geq1$ large enough satisfying $\|tu_0\|_{\mu}>\sigma$ and $I_{\mu}(tu_0)<0$.
Taking $e=tu_0$, the proof of Lemma \ref{lem4.4} is complete.
\end{proof}

Note that function $I_\mu$ does not satisfy the $(PS)_c$ condition
for every $c>0$. Therefore, we can find a special finite-dimensional
subspace to construct sufficiently small minimax levels. Next, we
obtain by assumption $(V_1)$ that there is $\xi_0\in \mathbb{H}^{N}$
satisfying $V(\xi_0)=\min_{\xi\in \mathbb{H}^{N}}V(\xi)=0$. In
general, we set $\xi_0=0$. By conditions $(M)$ and $(f_{2})'$, for
$u\in S_\mu$, one has
\begin{align*}
\begin{split}
I_{\mu}(u)&\leq\frac{m_1}{p\tau}\|u\|_{\mu}^{p\tau}-\frac{1}{2Q_{\lambda}^{\ast}}\int_{\mathbb{H}^{N}}\int_{\mathbb{H}^N}
\frac{|u(\eta)|^{Q_{\lambda}^{\ast}}|u(\xi)|^{Q_{\lambda}^{\ast}}}{|\eta^{-1}\xi|^{\lambda}}d\eta d\xi-a_0\int_{\mathbb{H}^{N}}|u|^{q_2}d\xi\\
&\leq\frac{m_1}{p\tau}\|u\|_{\mu}^{p\tau}-a_0\int_{\mathbb{H}^{N}}|u|^{q_2}d\xi.
\end{split}
\end{align*}
We define the function $J_\mu: S_\mu\rightarrow\mathbb{R}$ as follows:
$$J_\mu(u)=\frac{m_1}{p\tau}\|u\|_{\mu}^{p\tau}-a_0\int_{\mathbb{H}^{N}}|u|^{q_2}d\xi.$$
Thus, $I_{\mu}(u)\leq J_\mu(u)$, and we only need to construct small
minimax levels of $J_\mu(u)$. For any $0<\chi<1$, we choose
$\delta_\chi\in C_{0}^{\infty}(\mathbb{H}^N)$ with
$|\delta_\chi|_{q_2}=1$ and $\mbox{supp}\delta_\chi\subset
B_{r_\chi}(0)$ satisfying $[\delta_\chi]_{s,p}^{p}<\chi$. In the
sequel, we shall make a scaling argument. Lettting
$$e_\mu=\delta_\chi(\mu^{-\frac{\tau Q_{\lambda}^{\ast}}{Q(Q_{\lambda}^{\ast}-\tau p)}}\xi),$$
we have $\mbox{supp}e_\mu\subset B_{\mu^{\frac{\tau Q_{\lambda}^{\ast}}{Q(Q_{\lambda}^{\ast}-\tau p)}}r_\chi}(0)$.
Thus, for $\mu\in(0,1)$, $\tau>1$ and $t\geq0$, one has
\begin{align*}
\begin{split}
J_{\mu}(te_\mu)&=\frac{t^{\tau p}}{p\tau}\|e_\mu\|_{\mu}^{p\tau}-a_0t^{q_2}\int_{\mathbb{H}^{N}}|e_\mu|^{q_2}d\xi\\
&\leq\mu^{\frac{\tau Q_{\lambda}^{\ast}}{Q_{\lambda}^{\ast}-\tau
p}}\left[\frac{t^{\tau
p}}{p\tau}\left(\int\int_{\mathbb{H}^{2N}}\frac{|\delta_\chi(\xi)-\delta_\chi(\eta)|^{p}}{|\eta^{-1}\xi|^{N+ps}}d\xi
d\eta+\int_{\mathbb{H}^{N}}
V(\mu^{\frac{\tau Q_{\lambda}^{\ast}}{Q(Q_{\lambda}^{\ast}-\tau p)}}\xi)|\delta_\chi|^{p}d\xi\right)^{\tau} \right.\\
&\quad
\left.-a_0t^{q_2}\int_{\mathbb{H}^{N}}|\delta_\chi|^{q_2}d\xi\right]=\mu^{\frac{\tau
Q_{\lambda}^{\ast}}{Q_{\lambda}^{\ast}-\tau
p}}\Psi_\mu(t\delta_\chi).
\end{split}
\end{align*}
We define $\Psi_\mu\in C_1(S_\mu,\mathbb{H}^{N})$ as follows:
$$\Psi_\mu(u)=\frac{1}{p\tau}\left(\int\int_{\mathbb{H}^{2N}}\frac{|u(\xi)-u(\eta)|^{p}}{|\eta^{-1}\xi|^{N+ps}}d\xi d\eta+\int_{\mathbb{H}^{N}}
V(\mu^{\frac{\tau Q_{\lambda}^{\ast}}{Q(Q_{\lambda}^{\ast}-\tau
p)}}\xi)|u|^{p}d\xi\right)^{\tau}
-a_0\int_{\mathbb{H}^{N}}|u|^{q_2}d\xi$$ for every $u\in S_\mu$.
It is clear that
$$\max_{t\geq0}\Psi_\mu(t\delta_\chi)=\frac{q_2-\tau p}{q_2\tau p(q_2a_0)^{\frac{\tau p}{q_2-\tau p}}}\left(\int\int_{\mathbb{H}^{2N}}\frac{|\delta_\chi(\xi)-\delta_\chi(\eta)|^{p}}{|\eta^{-1}\xi|^{N+ps}}d\xi d\eta+\int_{\mathbb{H}^{N}}
V(\mu^{\frac{\tau Q_{\lambda}^{\ast}}{Q(Q_{\lambda}^{\ast}-\tau
p)}}\xi)|\delta_\chi|^{p}d\xi\right)^{\frac{\tau q_2}{q_2-\tau
p}}.$$ Due to condition $(M)$, we know that $V(0)=0$ and
$V\in(\mathbb{H}^{N}, \mathbb{R})$, so
 there is a constant
$\Lambda_\chi>0$ satisfying
$$0\leq V(\mu^{\frac{\tau Q_{\lambda}^{\ast}}{Q(Q_{\lambda}^{\ast}-\tau p)}}\xi)\leq\frac{\chi}{\delta_\chi},$$
for every $|\xi|\leq r_\chi$ and $0<\mu\leq\Lambda_\chi$. Since
$[\delta_\chi]_{s,p}^{p}<\chi$, we have
$$\max_{t\geq0}\Psi_\mu(t\delta_\chi)\leq\frac{q_2-\tau p}{q_2\tau p(q_2a_0)^{\frac{\tau p}{q_2-\tau p}}}(2\chi)^{\frac{\tau q_2}{q_2-\tau p}},$$
therefore
\begin{equation}\label{4.23}
\max_{t\geq0}I_\mu(t\delta_\chi)\leq\frac{q_2-\tau p}{q_2\tau
p(q_2a_0)^{\frac{\tau p}{q_2-\tau p}}}(2\chi)^{\frac{\tau
q_2}{q_2-\tau p}} \mu^{\frac{\tau
Q_{\lambda}^{\ast}}{Q_{\lambda}^{\ast}-\tau p}},
\end{equation}
for every $\mu\in(0,\Lambda_\chi]$. To be more precise, we state the following lemma.

\begin{lemma}\label{lem4.5}
Under conditions of Lemma \ref{lem4.3}, there
exists
a constant $\Lambda>0$ satisfying the following hypotheses:
for every fix $\mu\in(0,\Lambda)$, there is $\widetilde{e_\mu}\in S_\mu$ with $\|\widetilde{e_\mu}\|_{\mu}>\sigma$ satisfying $I_{\mu}(\widetilde{e_\mu})<0$
and
$$\max_{t\in[0,1]}I_\mu(t\widetilde{e_\mu})\leq\rho\mu^{\frac{\tau Q_{\lambda}^{\ast}}{Q_{\lambda}^{\ast}-\tau p}},$$
where $\rho=\frac{1}{\tau p}-\frac{1}{2Q_{\lambda}^{\ast}}+\frac{1}{q_0}(1-\frac{m_1}{m_0})$.
\end{lemma}

\begin{proof}
Let $\chi$ be small enough such that
$$\frac{q_2-\tau p}{q_2\tau p(q_2a_0)^{\frac{\tau p}{q_2-\tau p}}}(2\chi)^{\frac{\tau q_2}{q_2-\tau p}}\leq\rho.$$
For all $t\geq t_1$, take $\Lambda=\Lambda_\chi$ and choose
$t_1>0$ satisfying $t_1\|e_{\mu}\|_\mu>\sigma$ and
$I_\mu(te_\mu)<0$. Letting $\widetilde{e_\mu}=t_1e_{\mu}$, we can
obtain the desired result. This completes the proof of Lemma \ref{lem4.5}.
\end{proof}

For any $m\in\mathbb{N}$, $1\leq i\neq j\leq m$, select
 functions $\delta_\chi^{i}\in C_{0}^{\infty}(\mathbb{H}^N)$
satisfying $\mbox{supp}\delta_\chi^{i}\bigcap\mbox{supp}\delta_\chi^{j}=\emptyset$, $|\delta_\chi|_{q_2}=1$ and
$[\delta_\chi]_{s,p}^{p}<\chi$. There is $r_\chi^{m}>0$ satisfying $\mbox{supp}\delta_\chi^{i}\subset B_{r_\chi^{m}}(0)$ for
$i=1,2,\cdots, m$:
$$e_\mu=\delta_\chi(\mu^{-\frac{\tau Q_{\lambda}^{\ast}}{Q(Q_{\lambda}^{\ast}-\tau p)}}\xi)$$
and
$$E_{\mu,\chi}^{m}=\mbox{span}\{e_\mu^{1}, e_\mu^{2}, \cdots, e_\mu^{m}\}.$$
Note that $u=\sum_{i=1}^{m}c^{i}e_\mu^{i}\in E_{\mu,\chi}^{m}$, so we
obtain
$$\int\int_{\mathbb{H}^{2N}}\frac{|u_n(\xi)-u_n(\eta)|^{p}}{|\eta^{-1}\xi|^{N+ps}}d\xi d\eta=\sum_{i=1}^{m}|c^{i}|^{p}
\int\int_{\mathbb{H}^{2N}}\frac{|e_\mu^{i}(\xi)-e_\mu^{i}(\eta)|^{p}}{|\eta^{-1}\xi|^{N+ps}}d\xi
d\eta,$$

$$\int_{\mathbb{H}^N}V(\xi)|u|^pd\xi=\sum_{i=1}^{m}|c^{i}|^{p}\int_{\mathbb{H}^N}V(\xi)|e_\mu^{i}|^pd\xi,$$

$$\frac{1}{2Q_{\lambda}^{\ast}}\int_{\mathbb{H}^{N}}\int_{\mathbb{H}^N}\frac{|u(\eta)|^
{Q_{\lambda}^{\ast}}|u(\xi)|^{Q_{\lambda}^{\ast}}}{|\eta^{-1}\xi|^{\lambda}}d\eta
d\xi=\frac{1}{2Q_{\lambda}^{\ast}}
\sum_{i=1}^{m}|c^{i}|^{2Q_{\lambda}^{\ast}}\int_{\mathbb{H}^{N}}\int_{\mathbb{H}^N}\frac{|e_\mu^{i}(\eta)|^
{Q_{\lambda}^{\ast}}|e_\mu^{i}(\xi)|^{Q_{\lambda}^{\ast}}}{|\eta^{-1}\xi|^{\lambda}}d\eta
d\xi$$ and
$$\int_{\mathbb{H}^{N}}F(\xi,u)d\xi=\sum_{i=1}^{m}\int_{\mathbb{H}^{N}}F(\xi,c^{i}e_\mu^{i})d\xi.$$
Thus
$$I_{\mu}(u)=\sum_{i=1}^{m}I_{\mu}(c^{i}e_\mu^{i})$$
and
$$I_{\mu}(c^{i}e_\mu^{i})\leq\mu^{\frac{\tau Q_{\lambda}^{\ast}}{Q_{\lambda}^{\ast}-\tau p}}\Psi_\mu(c^{i}e_\mu^{i}).$$
Take $\beta=\max|\delta_\chi^{i}|_{p}^{p}$, $i=1, 2, \cdots, m$
and for every $|x|\leq r_{\chi}^{m}$ and $\mu\leq\Lambda_{m,\chi}.$
There exists
$\Lambda_{m,\chi}>0$ satisfying $V(\mu^{\frac{\tau Q_{\lambda}^{\ast}}{N(Q_{\lambda}^{\ast}-\tau p)}}\xi)\leq\frac{\chi}{\beta}$.
From \eqref{4.23}, we derive that
\begin{equation*}
\max_{u\in E_{\mu,\chi}^{m}}I_\mu(u)\leq\frac{q_2-\tau p}{q_2\tau p(q_2a_0)^{\frac{\tau p}{q_2-\tau p}}}(2\chi)^{\frac{\tau q_2}{q_2-\tau p}}
\mu^{\frac{\tau Q_{\lambda}^{\ast}}{Q_{\lambda}^{\ast}-\tau p}}
\end{equation*}
for every $\mu\in(0,\Lambda_{m,\chi}]$. Thus, we have the following
lemma.
\begin{lemma}\label{lem4.6}
Under the conditions
 of Lemma \ref{lem4.3}, for every $m\in\mathbb{N}$, there is $\Lambda_m>0$ satisfying the following hypothesis:
for every $\mu\in(0,\Lambda_{m}]$, there is an $m$-dimensional subspace $E_{\mu}^{m}$ satisfying
\begin{equation}\label{4.24}
\max_{u\in E_{\mu}^{m}}I_\mu(u)\leq\rho
\mu^{\frac{\tau Q_{\lambda}^{\ast}}{Q_{\lambda}^{\ast}-\tau p}}.
\end{equation}
\end{lemma}

\begin{proof}
Let $\chi$ small enough to satisfy
\begin{equation*}
\frac{q_2-\tau p}{q_2\tau p(q_2a_0)^{\frac{\tau p}{q_2-\tau p}}}(2\chi)^{\frac{\tau q_2}{q_2-\tau p}}\leq\rho
\end{equation*}
and choose $E_{\mu}^{m}=E_{\mu,\chi}^{m}$.
Then we obtain the desired result. This completes the proof of Lemma \ref{lem4.6}.
\end{proof}

\begin{proof}[Proof of Theorem \ref{the1.2}]
Apply Lemma \ref{lem4.5} and for every $\mu>0$, consider the
function $I_\mu$, let $\mu^{*}=\Lambda_\chi$ and define for every
$\mu\leq\mu^{*}$ the min-max value
$$c^{\mu}=\inf_{h\in\Gamma_{\mu}}\max_{t\in [0,1]}I_{\mu}(h(t)),$$
where
$$\Gamma_{\mu}=\{h\in C([0,1],W_\mu): h(0)=0\ \mbox{and}\ h(1)=\widetilde{e_\mu}\}.$$
By Lemma \ref{lem4.3} and Lemma \ref{lem4.5}, one has $\alpha\leq
c_\mu<\rho\mu^{\frac{\tau
Q_{\lambda}^{\ast}}{Q_{\lambda}^{\ast}-\tau p}}$. By Lemma
\ref{lem4.2}, we know that $I_\mu$ satisfies the $(PS)_{c_\mu}$
condition, and we can deduce that there is $u_1\in S_\mu$ satisfying
$I_{\mu}(u_{1})\rightarrow c_\mu,\ I_{\mu}'(u_{1})\rightarrow 0$.
Therefore, $u_1$ is a solution of \eqref{2.1}. Since
 $u_1$ is
a critical point of $I_\mu$, we have
\begin{align*}
\begin{split}
\rho\mu^{\frac{\tau Q_{\lambda}^{\ast}}{Q_{\lambda}^{\ast}-\tau p}}&\geq I_\mu(u_1)=I_\mu(u_1)-\frac{1}{q_0}\langle I'_\mu(u_1),u_1\rangle\\
&\geq\left(\frac{m_0}{\tau p}-\frac{m_1}{q_0}\right)\|u_1\|_\mu^{\tau p}+\left(\frac{m_1}{q_0}-\frac{1}{2 Q_{\lambda}^{\ast}}
\int_{\mathbb{H}^{N}}\int_{\mathbb{H}^N}\frac{|u_1(\eta)|^
{Q_{\lambda}^{\ast}}|u_1(\xi)|^{Q_{\lambda}^{\ast}}}{|\eta^{-1}\xi|^{\lambda}}d\eta d\xi\right),
\end{split}
\end{align*}
which
yields
  inequalities \eqref{1.2} and \eqref{1.3}.
This completes the proof of Theorem \ref{the1.2} (i). 

Next, we are going to prove Theorem \ref{the1.2} (ii).
We define
$$\Gamma=\{y\in C(S_\mu,S_\mu):\ y\ \mbox{is an odd homeomorphism}\},$$
and for every $B\in\Upsilon$, we define
$$i(B)=\min_{y\in\Gamma}\gamma(y(B)\bigcap\partial B_{\sigma}),$$
where $\sigma>0$ is a constant defined in Lemma \ref{lem4.3}.
Therefore, $i(B)$ is a version of the Benci pseudo-index  
(see Benci \cite{ben}). Let
$$c_{j}=\inf_{i(B)\leq j}\sup_{u\in B}I_\mu(u),\ j=1, 2, \cdots, m.$$
It is clear that
$$c_{1}\leq c_{2}\leq\cdots\leq c_{m}.$$
In the sequel, we are to going prove that $c_1\geq\alpha$ and
$c_m\leq \sup_{u\in E_\mu^{m}}I_\mu(u)$, where $\alpha>0$ is the
constant defined in Lemma \ref{lem4.3}. For all $B\in\Upsilon$, it
follows from Benci \cite[Theorem 1.4]{ben} that $i(B)\geq1$, so we
can deduce that $\gamma(B\bigcap\partial B_{\sigma})\geq1$. This
implies that $B\bigcap\partial B_{\sigma}\neq\emptyset$. By Lemma
\ref{lem4.3}, one has
$$I_\mu(u)>\alpha,\ \hbox{for every} \
 \|u\|_\mu=\sigma.$$
Thus $\sup_{u\in B}I_\mu(u)>\alpha$ and $c_1\geq\alpha$. Considering
that the Krasnoselskii genus satisfies the dimension property 
(see Benci
\cite{ben}), we obtain
$$\gamma(y(E_{\mu}^{m})\bigcap\partial B_{\sigma})=\dim(E_{\mu}^{m})=m,\ \mbox{for every}
\
 y\in\Gamma,$$
which
implies
 that $i(E_{\mu}^{m})=m$.
Hence, $c_m\leq\sup_{u\in E_{\mu}^{m}}I_\mu(u)$.
By \eqref{4.24}, one has
$$\alpha\leq c_{1}\leq c_{2}\leq\cdots\leq c_{m}\leq\sup_{u\in E_{\mu}^{m}}I_\mu(u)\leq \rho
\mu^{\frac{\tau Q_{\lambda}^{\ast}}{Q_{\lambda}^{\ast}-\tau p}},$$
where $\rho>0$ is a constant defined in Lemma \ref{lem4.2}. It can
be seen from Lemma \ref{lem4.2}, $I_\mu(u)$ satisfies the $(PS)_c$
condition at all levels $c\in(0,
\rho(m_0\mu^{\tau}H_{Q_{\lambda}^{\ast}}^{\tau})^{\frac{Q_{\lambda}^{\ast}}{Q_{\lambda}^{\ast}-\tau
p}})$. Finally, by using the general critical point theory, we
obtain that all $c_j$ of $1\leq j\leq m$ are critical values of
$I_\mu(u)$. Since $I_\mu(u)$ is even, $I_\mu(u)$ has at least $m$
pairs of critical points. Therefore, $I_\mu(u)$ has at least $m$
pairs of critical points as the solutions of problem \eqref{1.1}.
The proof of Theorem \ref{the1.2} is thus complete. \
\end{proof}

\subsection*{\bf Acknowledgements}
We thank the referee for comments and suggestions.
Y. Song was supported by the National Natural Science
Foundation of China (Grant No.12001061), the Research Foundation of
Department of Education of Jilin Province (Grant No. JJKH20220822KJ),
the Natural Science Foundation of Jilin Province (Grant No.  
222614JC010793935) and Innovation and Entrepreneurship Talent
Funding Project of Jilin Province (Grant No. 2023QN21). 
D.D. Repov\v{s} was supported by the Slovenian  Research
and Innovation
Agency program No. P1-0292 and grants Nos. J1-4031, J1-4001,
N1-0278, N1-0114, and N1-0083.

\end{document}